\documentclass[10pt, a4paper,reqno]{amsart}
\usepackage{hyperref}
\hypersetup{colorlinks=true, bookmarks=true, citecolor=blue}
\usepackage{amssymb,amstext,amsmath,amscd,amsthm,amsfonts,enumerate,graphicx,latexsym,color,cleveref}
\usepackage[all]{xy}
\usepackage{comment}

\allowdisplaybreaks

\newtheorem*{thma}{Theorem A}
\newtheorem*{thmb}{Theorem B}
\newtheorem*{thmc}{Theorem C}

\usepackage{aliascnt}
\usepackage{cleveref}

\newtheorem{thm}{Theorem}[section]

\newaliascnt{cor}{thm}
\newtheorem{cor}[cor]{Corollary}
\aliascntresetthe{cor}

\newaliascnt{lem}{thm}
\newtheorem{lem}[lem]{Lemma}
\aliascntresetthe{lem}

\newaliascnt{prop}{thm}
\newtheorem{prop}[prop]{Proposition}
\aliascntresetthe{prop}

\theoremstyle{definition}

\newaliascnt{conv}{thm}
\newtheorem{conv}[conv]{Convention}
\aliascntresetthe{conv}

\newaliascnt{setup}{thm}

\aliascntresetthe{setup}

\newaliascnt{dfn}{thm}
\newtheorem{dfn}[dfn]{Definition}
\aliascntresetthe{dfn}

\newaliascnt{rmk}{thm}
\newtheorem{rmk}[rmk]{Remark}
\aliascntresetthe{rmk}

\newaliascnt{ques}{thm}
\newtheorem{ques}[ques]{Question}
\aliascntresetthe{ques}

\newaliascnt{conj}{thm}

\aliascntresetthe{conj}

\newtheorem*{conj*}{Conjecture}

\newaliascnt{eg}{thm}
\newtheorem{eg}[eg]{Example}
\aliascntresetthe{eg}

\newaliascnt{nota}{thm}
\newtheorem{nota}[nota]{Notation}
\aliascntresetthe{nota}

\theoremstyle{remark}

\newtheorem*{ac}{Acknowledgments}

\newtheorem*{claim*}{Claim}

\crefname{thm}{theorem}{theorems}
\Crefname{thm}{Theorem}{Theorems}

\crefname{cor}{corollary}{corollaries}
\Crefname{cor}{Corollary}{Corollaries}

\crefname{lem}{lemma}{lemmas}
\Crefname{lem}{Lemma}{Lemmas}

\crefname{prop}{proposition}{propositions}
\Crefname{prop}{Proposition}{Propositions}

\crefname{conv}{convention}{conventions}
\Crefname{conv}{Convention}{Conventions}

\crefname{dfn}{definition}{definitions}
\Crefname{dfn}{Definition}{Definitions}

\crefname{rem}{remark}{remarks}
\Crefname{rem}{Remark}{Remarks}

\crefname{ques}{question}{questions}
\Crefname{ques}{Question}{Questions}

\crefname{conj}{conjecture}{conjectures}
\Crefname{conj}{Conjecture}{Conjectures}

\crefname{eg}{example}{examples}
\Crefname{eg}{Example}{Examples}

\crefname{nota}{notation}{notations}
\Crefname{nota}{Notation}{Notations}

\crefname{claim}{claim}{claims}
\Crefname{claim}{Claim}{Claims}

\numberwithin{equation}{thm}
\def\FF{\mathbb{F}}
\def\GG{\mathbb{G}}

\def\NN{\mathbb{N}}

\def\RR{\mathbb{R}}

\def\ZZ{\mathbb{Z}}

\def\cA{\mathcal{A}}

\def\cG{\mathcal{G}}
\def\cP{\mathcal{P}}
\def\cT{\mathcal{T}}
\def\cX{\mathcal{X}}

\def\fm{\mathfrak{m}}

\def\fp{\mathfrak{p}}
\def\fq{\mathfrak{q}}

\def\sB{\mathsf{B}}
\def\sH{\mathsf{H}}

\def\c{\mathsf{c}}

\def\ann{\operatorname{ann}}

\def\dim{\operatorname{\mathsf{dim}}}
\def\depth{\operatorname{\mathsf{depth}}}
\def\width{\operatorname{\mathsf{width}}}
\def\Ext{\operatorname{\mathsf{Ext}}}

\def\height{\operatorname{\mathsf{ht}}}
\def\Hom{\operatorname{\mathsf{Hom}}}

\def\mod{\operatorname{\mathsf{mod}}}
\def\Mod{\operatorname{\mathsf{Mod}}}
\def\Tor{\operatorname{\mathsf{Tor}}}
\def\Ker{\operatorname{\mathsf{Ker}}}
\def\Im{\operatorname{\mathsf{Im}}}
\def\Spec{\operatorname{\mathsf{Spec}}}
\def\Supp{\operatorname{\mathsf{Supp}}}
\def\Tot{\operatorname{\mathsf{Tot}}}

\def\D{\mathsf{D}}

\def\Db{\mathsf{D_b}}

\def\Dfb{\mathsf{D^f_b}}

\def\Dfp{\mathsf{D^f_+}}

\def\Pf{\mathsf{D^{pf}}}

\def\proj{\operatorname{\mathsf{proj}}}
\def\ltensor{\otimes^{\mathbb{L}}}
\def\rHom{\mathbb{R}\mathsf{Hom}}
\def\sup{\operatorname{\mathsf{sup}}}
\def\hsup{\operatorname{\mathsf{hsup}}}
\def\inf{\operatorname{\mathsf{inf}}}
\def\hinf{\operatorname{\mathsf{hinf}}}
\def\range{\operatorname{\mathsf{range}}}

\def\thick{\mathsf{thick}}

\def\cx{\mathsf{cx}}
\def\Cone{\mathsf{Cone}}
\def\coCone{\mathsf{coCone}}

\def\CIdim{\operatorname{\mathsf{CIdim}}}

\def\Gdim{\mathsf{Gdim}}
\def\pd{\mathsf{pd}}
\def\id{\operatorname{\mathsf{id}}}

\def\qpd{\operatorname{\mathsf{qpd}}}
\def\qpl{\operatorname{\mathsf{qpl}}}

\def\op{\mathsf{op}}

\newcommand{\mdot}{%
  \mathbin{\vcenter{\hbox{\scalebox{0.5}{$\bullet$}}}}%
}
\begin{document}

\title[Quasi-projective dimension for complexes via filtrations]{Quasi-projective dimension for complexes via filtrations}
\author{Hiroki Matsui}
\address{Department of Mathematics, Graduate School of Advanced Science and Engineering, Hiroshima University,
1-3-1 Kagamiyama, Higashi-Hiroshima, 739-8526, JAPAN}
\email{hrmatsui@hiroshima-u.ac.jp}
\urladdr{https://mthiroki.github.io/index.html}
\thanks{The author was partly supported by JSPS Grant-in-Aid for Scientific Research 26K06766.}

\begin{abstract}
We extend quasi-projective dimension and quasi-projective length from finitely generated modules to homologically finite complexes by using finite filtrations in the derived category. 
Our definitions recover the original invariants of Gheibi--Jorgensen--Takahashi for modules and behave well under exact functors, which simplifies the proofs of several results. 
We establish the Auslander--Buchsbaum formula, the derived depth and width formulas, and the dependency formula for complexes of finite quasi-projective dimension. 
Extending a result of Gheibi--Jorgensen--Takahashi, we show that every homologically finite complex has finite quasi-projective dimension over a suitable complete intersection ring.
We also prove a new intersection theorem and a descent theorem for Serre's conditions. 
Finally, we obtain vanishing results for Tor, Ext, and Tate (co)homology, and also symmetry of eventual Ext vanishing.
\end{abstract}
\maketitle
\tableofcontents
\section{Introduction}

Throughout the paper, unless otherwise specified, $R$ denotes a commutative noetherian ring.

Homological dimensions measure how far a module or a complex is from being projective and often reflect properties of the underlying ring.
For example, a commutative noetherian local ring is regular (resp.~Gorenstein, complete intersection) if and only if every finitely generated module has finite projective dimension (resp.~finite Gorenstein dimension, finite complete intersection dimension); see \cite{AB,AGP,Gbook,CFH,SW}.
Several homological dimensions, such as Gorenstein dimension, generalize projective dimension by enlarging the class of objects used in resolutions.

Quasi-projective dimension, introduced by Gheibi, Jorgensen, and Takahashi in \cite{GJT}, takes a different approach.
Instead of enlarging the class of resolving objects, it weakens the notion of a resolution.
They showed, among other results, that modules of finite quasi-projective dimension satisfy an Auslander--Buchsbaum formula and a depth formula, and that every finitely generated module over a quotient of a regular local ring by a regular sequence has finite quasi-projective dimension.
Since then, various aspects of quasi-projective dimension have been investigated by several authors \cite{FL,JPMMR}.
These studies suggest that modules of finite quasi-projective dimension behave in many respects like modules over complete intersection local rings, or more generally, like modules of finite complete intersection dimension.

The purpose of this paper is to extend quasi-projective dimension from modules to homologically finite complexes and to study its consequences.
Recall that a complex is {\em homologically finite} if its homology is bounded and finitely generated.
For complexes, however, knowing only the graded homology of a projective complex is generally not enough to recover all the information contained in the original complex.
Indeed, complexes with isomorphic graded homology need not be isomorphic in the derived category and may have different homological properties.
We therefore define {\em quasi-projective dimension} $\qpd_R(M)$ and {\em quasi-projective length} $\qpl_R(M)$ using filtrations in the derived category.
This point of view is in the spirit of resolutions of DG modules and of reducing complexity (cf. \cite{Ber1, Ber2, Ker}).
For finitely generated modules, our definitions of quasi-projective dimension and quasi-projective length agree with those of \cite{GJT}.
One advantage of the filtration formulation is that it behaves well under exact functors.
This allows us to prove several formulas and vanishing results by uniform arguments.

Our first main result gives the basic numerical properties of these invariants.
These results extend those of \cite{FL, GJT}; see also \cite{JPMMR}.

\begin{thm}[\Cref{AB}, \Cref{DDF}, \Cref{DWF}]
Let $R$ be a noetherian local ring.
Let $M$ be a homologically finite complex and assume that $\qpd_R(M)<\infty$.
Then the following assertions hold.
\begin{enumerate}[\rm(1)]
\item
{\rm(Auslander--Buchsbaum formula)}
One has
\[
\qpd_R(M)=\depth(R)-\depth_R(M).
\]

\item
{\rm(Derived depth formula)}
For a homologically finite complex $N$ with $M\ltensor_RN$ homologically bounded, one has
\[
\depth_R(M\ltensor_RN)
=
\depth_R(M)+\depth_R(N)-\depth(R).
\]

\item
{\rm(Derived width formula)}
For a homologically finite complex $N$ with $\rHom_R(M,N)$ homologically bounded, one has
\[
\width_R(\rHom_R(M,N))
=
\depth_R(M)+\width_R(N)-\depth(R).
\]
\end{enumerate}
\end{thm}
\noindent
The derived depth formula also yields a dependency formula; see \Cref{dep}.

As mentioned above, complexes of finite quasi-projective dimension are expected to behave similarly to complexes over complete intersections.
Indeed, over a quotient of a regular local ring with infinite residue field by a regular sequence, every homologically finite complex has finite quasi-projective dimension.

\begin{thm}[\Cref{cxbound}]
Suppose that $R=Q/(x_1,\ldots,x_c)$, where $Q$ is a regular local ring with infinite residue field and $x_1,\ldots,x_c$ is a $Q$-regular sequence.
Then every non-acyclic homollogically finite complex $M$ has finite quasi-projective dimension.
Moreover,
\[
\qpl_R(M)-\qpd_R(M)\leq\cx_R(M).
\]
\end{thm}
\noindent
Here, $\cx_R(M)$ denotes the {\em complexity} of $M$; see \cite{IFR,AGP}.
Thus the difference $\c_R(M) := \qpl_R(M)-\qpd_R(M)$ behaves, in several respects, like complexity.
The role played by $\c_R(M)$ in our results parallels that of complexity for complexes of finite complete intersection dimension.
Sharif and Yassemi proved an extension of the new intersection theorem for complexes of finite complete intersection dimension in \cite{SY}, while Celikbas and Piepmeyer \cite{CP} proved a descent theorem for Serre's conditions.
We obtain corresponding results for complexes of finite quasi-projective dimension, which are new even in the module case.

\begin{thm}[\Cref{NIT}, \Cref{Serre}]
Let $R$ be a noetherian local ring.
Let $M,N$ be homologically finite complexes and assume that $M$ is non-acyclic and $\qpd_R(M)<\infty$.
Then the following assertions hold.
\begin{enumerate}[\rm(1)]
\item
{\rm(New intersection theorem)}
One has
\[
\dim_R(N)
\leq
\dim_R(M\ltensor_RN)
+\qpd_R(M)+\c_R(M).
\]

\item {\rm(Descent theorem for Serre's condition)}
Let $n\geq0$.
If $M\ltensor_RN$ is homologically bounded and satisfies Serre's condition
$(S_{n+\c_R(M)})$, then $N$ satisfies $(S_n)$.
\end{enumerate}
\end{thm}

These results suggest a close relationship between $\c_R(M)$ and $\cx_R(M)$.
In particular, it is natural to ask whether every non-zero homologically finite complex of finite complete intersection dimension has finite quasi-projective dimension and satisfies
\[
\c_R(M)=\cx_R(M).
\]

Finally, we study vanishing of Tor, Ext, and Tate (co)homology.
Over a Gorenstein local ring, symmetry of eventual Ext vanishing for finitely generated modules of finite complete intersection dimension is obtained by J{\o}rgensen's theorem \cite[Theorem 4.1]{Jor}.
Gheibi, Jorgensen, and Takahashi \cite[Theorem 6.16]{GJT} proved an analogous symmetry result for finitely generated modules when either module has finite quasi-projective dimension.
More recently, Ajani, Martins, Mendoza Rubio, and Soto Levins \cite[Corollary 5.4]{AMMRSL} further extended this result to noetherian rings admitting a dualizing complex, obtaining a symmetry statement.
The following theorem extends their result by allowing $M$ to be a homologically finite complex.


\begin{thm}[\Cref{symm0}]
Let $R$ be a noetherian ring admitting a dualizing complex $D_R$ (for example, if $R$ is complete local).
Let $M, N \in \Dfb(R)$ such that $\qpd_R(M) < \infty$ and $\Gdim_R(M) < \infty$.
Then 
\[
\Ext_R^{\gg 0}(M,N) \cong 0 \Longleftrightarrow \Ext_R^{\gg 0}(N, M \ltensor_R D_R) \cong 0.
\]
\end{thm}
\noindent
Here, $\Gdim_R(M)$ denotes the Gorenstein dimension of $M$.

\noindent

Recently, Chen, Hu, and Yang \cite{CHY} introduced another extension of quasi-projective dimension to complexes by prescribing the graded homology of a projective complex.
Their invariant and ours agree for modules, but they differ for complexes.
In Example~\ref{eg:homology}, we give an example illustrating this difference and showing that the Auslander--Buchsbaum formula in \cite[Theorem~4.4]{CHY} does not hold as stated.

The paper is organized as follows.
Section~2 collects the basic facts on finite filtrations.
Section~3 introduces quasi-projective dimension and quasi-projective length for complexes and establishes several basic properties. 
Among them, we prove that our definitions agree with those of Gheibi, Jorgensen, and Takahashi for modules. 
Section~4 proves the Auslander--Buchsbaum, derived depth, dependency, and derived width formulas.
Section~5 studies complete intersections, complexity, the new intersection theorem, and Serre's conditions.
Section~6 is devoted to vanishing of Ext, Tor, and Tate (co)homology and to symmetry of Ext vanishing over Gorenstein rings.

\section{Preliminaries}

In this section, we recall several basic definitions and properties.
We also introduce the notion of filtration in a triangulated category that will be used throughout the paper.
For standard notions and results from homological algebra, we refer the reader to \cite{CFH,Wei}.

\subsection{Conventions and basic facts}

\begin{conv}\label{conv:general}
For a triangulated category $\cT$, we denote the $n$-th shift functor by $[n]$.
For a morphism $f\colon X\to Y$ in $\cT$, we write $\Cone(f)$ for a {\em cone} of $f$, so that there is an exact triangle
\[
X \xrightarrow{f} Y \to \Cone(f) \to X[1].
\]
We define the {\em cocone} of $f$ by $\coCone(f):=\Cone(f)[-1]$.
Thus, rotating the above triangle gives an exact triangle
\[
\coCone(f) \to X \xrightarrow{f} Y \to \Cone(f).
\]

Let $\cA$ be an abelian category with enough projective objects, and index all complexes homologically.
We write $\D(\cA)$ for the derived category of $\cA$, $\Db(\cA)$ for its full subcategory of homologically bounded complexes.
We use $\simeq$ to denote an isomorphism in the derived category.
For $X\in\D(\cA)$, define the {\em projective dimension} of $X$ by 
\[
\pd_\cA(X)
=
\inf\left\{
n \ \middle|\
\begin{array}{l}
P\simeq X \text{ in }\D(\cA)
\text{ for some bounded complex }P\text{ of projective objects},\\
P_i\cong 0 \text{ for all } i>n
\end{array}
\right\}.
\]
We write $\cP(\cA)$ for the full subcategory of $\D(\cA)$ consisting of complexes of finite projective dimension.

We write $\Mod R$ for the category of $R$-modules, $\mod R$ for its full subcategory of finitely generated modules, and set $\D(R) := \D(\Mod R)$.
A complex of $R$-modules is called {\em homologically finite} if it is homologically bounded and all its homology modules are finitely generated.
We write $\Dfb(R)$ for the full subcategory of $\D(R)$ consisting of homologically finite complexes, and set $\Pf(R) = \Dfb(R) \cap \cP(\Mod R)$.
An object of $\Pf(R)$ is called a {\em perfect complex}.
We note that there are triangle equivalences $\Dfb(R) \simeq \Db(\mod R)$ and $\Pf(R) \simeq \cP(\mod R) \simeq \mathsf{K^b}(\proj R)$.
Here, $\mathsf{K^b}(\proj R)$ is the homotopy category of bounded complexes of finitely generated projective $R$-modules.

For $M,N\in\Dfb(R)$ and $n \in \ZZ$ write 
\[
\Ext_R^n(M,N) := \sH_{-n}(\rHom_R(M,N)),\qquad
\Tor_n^R(M,N) :=\sH_n(M\ltensor_R N).
\]
\end{conv}

\begin{nota}\label{nota:invariants}
For a complex $X$, put
\[
\begin{aligned}
 \sup(X) &= \sup\{i\mid X_i\neq0\},&
 \inf(X) &= \inf\{i\mid X_i\neq0\},\\
 \hsup(X) &=\sup\{i\mid\sH_i(X)\neq0\},&
 \hinf(X) &=\inf\{i\mid\sH_i(X)\neq0\}.
\end{aligned}
\]
\end{nota}

We recall the following basic invariants and definitions for complexes over commutative noetherian rings.

\begin{dfn}\label{def:complex-invariants}
Let $X\in\D(R)$.
\begin{enumerate}[\rm(1)]
\item The {\em support} of $X$ is
\[
\Supp_R(X)=\{\fp\in\Spec R\mid X_\fp\not\simeq0\text{ in }\D(R_\fp)\}=\bigcup_{i\in\ZZ}\Supp_R\sH_i(X).
\]
The {\em dimension} of $X$ is
\[
\dim_R(X) :=\sup\{\dim(R/\fp)-\hinf X_\fp\mid \fp\in\Supp_R(X)\}.
\]
\item 
Assume $(R,\fm,k)$ is a commutative noetherian local ring.
The {\em depth} $\depth_R(X)$ and the {\em width} $\width_R(X)$ of $X$ are defined by
\begin{align*}
\depth_R(X) &:= -\hsup(\rHom_R(k,X)),	 \\
\width_R(X) &:= \hinf(k\ltensor_RX).
\end{align*}

\item
Let $n\geq0$ be an integer.
We say that $X$ satisfies {\em Serre's condition $(S_n)$} if
\[
\depth_{R_\fp}(X_\fp)+\hinf (X_\fp) \ge \min\{n,\operatorname{ht}\fp\}
\]
for all $\fp \in \Supp_R(X)$.
\end{enumerate}
\end{dfn}

\subsection{Filtrations}

The following definition plays a key role throughout this paper.

\begin{dfn}
Let $\cT$ be a triangulated 	category and let $M, X \in \cT$.
\begin{enumerate}[\rm(1)]
\item
A {\em finite filtration} of $X$ by $M$ in $\cT$ is a decreasing sequence
\[
\FF = (F^{t+1} \to F^{t} \to \cdots \to F^{s+1} \to F^s)
\]
in $\cT$ with $s \le t$ such that $F^{t+1} \cong 0$, $F^s \cong X$, and 
\[
\Cone(F^{i+1} \to F^i ) \cong M[i]^{\oplus a_i}\qquad(i=s,\ldots,t),
\]
for some non-negative integers $a_s,\ldots,a_t$ with $a_s,a_t > 0$.
The interval $[s,t]$ is called the {\em index interval} of $\FF$, and $\range(\FF) := t-s$ is called the {\em range} of $\FF$.

\item
A {\em finite cofiltration} of $X$ by $M$ in $\cT$ is an increasing sequence
\[
\GG = (G_{s} \to G_{s+1} \to \cdots \to G_{t} \to G_{t+1})
\]
in $\cT$ with $s \le t$ such that $G_{t+1} \cong 0$, $G_s \cong X$, and 
\[
\coCone(G_{i} \to G_{i+1}) \cong M[-i]^{\oplus a_i}\qquad(i=s,\ldots,t),
\]
for some non-negative integers $a_s,\ldots,a_t$ with $a_s,a_t > 0$.
The interval $[s,t]$ is called the {\em index interval} of $\GG$, and $\range(\GG) := t-s$ is called the {\em range} of $\GG$.
\end{enumerate}
\end{dfn}

\begin{rmk}
Equip $\cT^\op$ with its standard triangulated structure.
\begin{enumerate}[\rm(1)]
\item
A finite filtration (resp. a finite cofiltration) of $X$ by $M$ in $\cT$ is a finite cofiltration (resp. a finite filtration) of $X$ by $M$ in $\cT^\op$.
\item
Let $\Phi: \cT \to \cT'$ be an exact functor between triangulated categories.
If $\Phi$ is covariant, $\Phi$ sends a finite filtration (resp. a finite cofiltration) of $X$ by $M$ in $\cT$ to a finite filtration (resp. a finite cofiltration) of $\Phi(X)$ by $\Phi(M)$ in $\cT'$.
If $\Phi$ is contravariant, $\Phi$ sends a finite filtration (resp. a finite cofiltration) of $X$ by $M$ in $\cT$ to a finite cofiltration (resp. a finite filtration) of $\Phi(X)$ by $\Phi(M)$ in $\cT'$.
\end{enumerate}
\end{rmk}

The following lemma is elementary and will be used frequently.

\begin{lem}\label{thick}
Let $X,M \in \cT$ be non-zero objects and let $\cX \subseteq \cT$ be a thick subcategory.
Let $\FF = (F^{t+1} \to F^{t} \to \cdots \to F^{s+1} \to F^s)$ be a finite filtration of $X$ by $M$.
If $M \in \cX$, then $F^i \in \cX$ for all $i=s,\ldots,t$. 	
\end{lem}

\begin{proof}
There are exact triangles $F^{i+1} \to F^i \to M[i]^{\oplus a_i} \to F^{i+1}[1]$.
Since $F^{t+1} \cong 0 \in \cX$, backward induction on $i$ shows $F^i \in \cX$ for each $i$. 
\end{proof}

We next recall the notions of homological and cohomological functors in order to describe the behavior of homological bounds under finite filtrations and cofiltrations.

\begin{dfn}
Let $\cT$ be a triangulated category and let $\cA$ be an abelian category.
\begin{enumerate}[\rm(1)]
\item
A {\em homological functor}	$H_* = (H_n)_{n \in \ZZ}$ from $\cT$ to $\cA$ is a sequence of additive functors $H_n: \cT \to \cA$ satisfying the following conditions:
\begin{itemize}
\item
There is a natural isomorphism $H_{n}(M[i]) \cong H_{n-i}(M)$ for $M \in \cT$.
\item
For an exact triangle $L \xrightarrow{f} M \xrightarrow{g} N \xrightarrow{h} L[1]$, the induced sequence
\[
\cdots \to H_{n+1}(N) \xrightarrow{H_{n+1}(h)} H_n(L) \xrightarrow{H_{n}(f)} H_n(M) \xrightarrow{H_{n}(g)}  H_n(N) \xrightarrow{H_{n}(h)} H_{n-1}(L) \to \cdots 
\]
is exact.	
\end{itemize}

\item
A {\em cohomological functor} $H^* = (H^n)_{n \in \ZZ}$ from $\cT$ to $\cA$ is a sequence of additive functors $H^n: \cT \to \cA$ satisfying the following conditions:
\begin{itemize}
\item
There is a natural isomorphism $H^{n}(M[i]) \cong H^{n+i}(M)$ for $M \in \cT$.
\item
For an exact triangle $L \xrightarrow{f} M \xrightarrow{g} N \xrightarrow{h} L[1]$, the induced sequence
\[
\cdots \to H^{n-1}(N) \xrightarrow{H^{n-1}(h)} H^n(L) \xrightarrow{H^{n}(f)} H^n(M) \xrightarrow{H^{n}(g)}  H^n(N) \xrightarrow{H^{n}(h)} H^{n+1}(L) \to \cdots 
\]
is exact.	
\end{itemize}
\end{enumerate}
For $M\in\cT$, we use the notation
\[
\begin{aligned}
\inf (H_*(M)) &:=\inf\{n\in\ZZ\mid H_n(M)\not\cong0\},&
\sup (H_*(M)) &:=\sup\{n\in\ZZ\mid H_n(M)\not\cong0\},\\
\inf (H^*(M)) &:=\inf\{n\in\ZZ\mid H^n(M)\not\cong0\},&
\sup (H^*(M)) &:=\sup\{n\in\ZZ\mid H^n(M)\not\cong0\}.
\end{aligned}
\]
\end{dfn}

\begin{rmk}
Let $\cA$ be an abelian category.
By the basic properties of derived categories, the homology functors $\sH_n: \D(\cA) \to \cA$ define a homological functor $\sH_* = (\sH_n)_{n \in \ZZ}$.
Thus $\hinf(M) = \inf(\sH_*(M))$ and $\hsup(M) = \sup(\sH_*(M))$ hold.
\end{rmk}

\begin{lem}\label{bound}
Let $\FF = (F^{t+1} \to F^{t} \to \cdots \to F^{s+1} \to F^s)$ be a finite filtration of $X$ by $M$ in $\cT$.
\begin{enumerate}[\rm(1)]
\item	
Let $H_*$ be a homological functor from $\cT$ to $\cA$.
Then the following statements hold.
\begin{enumerate}[\rm(i)]
\item
If $\inf(H_*(M))> -\infty$, then $\inf(H_*(X)) = \inf(H_*(M)) + s$
\item
If $\sup(H_*(M))< \infty$, then $\sup(H_*(X)) = \sup(H_*(M)) + t$
\end{enumerate}

\item	
Let $H^*$ be a cohomological functor from $\cT$ to $\cA$.
Then the following statements hold.
\begin{enumerate}[\rm(i)]
\item
If $\inf(H^*(M))> -\infty$, then $\inf(H^*(X)) = \inf(H^*(M)) - t$
\item
If $\sup(H^*(M))< \infty$, then $\sup(H^*(X)) = \sup(H^*(M)) - s$
\end{enumerate}
\end{enumerate}
\end{lem}

\begin{proof}
If $H_*(M) \cong 0$, all the assertions are immediate by \Cref{thick}.
Thus, we may assume $H_*(M) \not\cong 0$.

(1)(i)
Put $u = \inf(H_*(M))$.
From the exact triangle $F^{i+1} \to F^{i} \to M[i]^{\oplus a_i} \to F^{i+1}[1]$	, there is an exact sequence
\begin{align}\label{eq:bound}
\cdots \to H_n(F^{i+1}) \to H_n(F^{i}) \to H_{n-i}(M)^{\oplus a_i} \to H_{n-1}(F^{i+1}) \to \cdots. 
\end{align}
This exact sequence yields
\[
\inf(H_*(F^i)) \ge \min\{\inf(H_*(F^{i+1})), \inf(H_{*-i}(M)^{\oplus a_i})\} \ge \min\{\inf(H_*(F^{i+1})), u+i\}.
\]
Moreover, as $F^{t} \cong M[t]^{\oplus a_t}$ and $a_t > 0$, we have $\inf(H_*(F^t)) = u+t$.
By descending induction, we obtain $\inf(H_*(F^i)) \ge u+i$ and in particular $\inf(H_*(F^s)) \ge u+s$.
To show the equality, it suffices to prove $H_{u+s}(F^s) \not\cong 0$.
Since $\inf(H_{*}(F^{s+1})) \ge u+s+1 > u+s$, there is an isomorphism $H_{u+s}(F^s) \cong H_{u}(M)^{\oplus a_s}$.
The right-hand side is non-zero because $a_s > 0$.
Hence, we get $H_{u+s}(F^s) \not\cong 0$ and thus, we conclude $\inf(H_*(X)) = \inf(H_*(F^s)) = u+s$.

(ii)
Put $v = \sup(H_*(M))$.
From the exact sequence (\ref{eq:bound}), one has 
\[
\sup(H_*(F^i)) \le \max\{\sup(H_*(F^{i+1})), \sup(H_{*-i}(M)^{\oplus a_i})\} \le \max\{\sup(H_*(F^{i+1})), v+i\}.
\]
As above, $\sup(H_*(F^t)) = v+t$.
By descending induction, we obtain $\sup(H_*(F^i)) \le v+t$ and in particular $\sup(H_*(F^s)) \le v+t$.
Thus, it remains to show $H_{v+t}(F^s) \not\cong 0$.
Since $\sup(H_{*-i}(M)) = v+i < v+t$ for $i<t$, we have isomorphisms
\[
H_{v+t}(F^s) \cong H_{v+t}(F^{s+1}) \cong \cdots \cong H_{v+t}(F^{t}) \cong H_v(M)^{\oplus a_t}.
\] 
The rightmost term is non-zero because $a_t > 0$, so $H_{v+t}(F^s) \not\cong 0$.
This shows the desired equality $\sup(H_*(X)) = \sup(H_*(F^s)) = v + t$.

(2)
Setting $H_n := H^{-n}$, $H_*$ is a homological functor.
Thus, the result follows from (1):
\begin{align*}
\inf(H^*(X)) &= -\sup(H_*(X)) = - (\sup(H_*(M)) +t) = \inf(H^*(M)) -t \\
\sup(H^*(X)) &= - \inf(H_*(X)) = - (\inf(H_*(M)) +s) = \sup(H^*(M)) -s.
\end{align*}
\end{proof}

Since finite cofiltrations in $\cT$ are finite filtrations in $\cT^\op$, we have the following dual statement.

\begin{lem}\label{bound2}
Let $\GG = (G_{s} \to G_{s+1} \to \cdots \to G_{t} \to G_{t+1})$ be a finite cofiltration of $X$ by $M$ in $\cT$.
\begin{enumerate}[\rm(1)]
\item	
Let $H_*$ be a homological functor from $\cT$ to $\cA$.
Then the following statements hold.
\begin{enumerate}[\rm(i)]
\item
If $\inf(H_*(M))> -\infty$, then $\inf(H_*(X)) = \inf(H_*(M)) - t$
\item
If $\sup(H_*(M))< \infty$, then $\sup(H_*(X)) = \sup(H_*(M)) - s$
\end{enumerate}

\item	
Let $H^*$ be a cohomological functor from $\cT$ to $\cA$.
Then the following statements hold.
\begin{enumerate}[\rm(i)]
\item
If $\inf(H^*(M))> -\infty$, then $\inf(H^*(X)) = \inf(H^*(M)) + s$
\item
If $\sup(H^*(M))< \infty$, then $\sup(H^*(X)) = \sup(H^*(M)) + t$
\end{enumerate}
\end{enumerate}
\end{lem}

Over a local ring $(R,\mathfrak m,k)$, applying Lemmas~\ref{bound} and \ref{bound2} to the cohomological functor $\Ext^*_R(k,-)$ and the homological functor $\Tor_*^R(k,-)$, respectively, we obtain the following proposition.

\begin{prop}\label{dwlem}
Assume that $(R,\mathfrak m,k)$ is local.
\begin{enumerate}[\rm(1)]
\item
Let $M, X \in \Dfb(R)$ be non-zero objects.
Let $\FF$ be a finite filtration of $X$ by $M$ with index interval $[s,t]$.
Then $\depth_R(X) = \depth_R(M) - t$.

\item
Let $M, X \in \Dfb(R)$ be non-zero objects.
Let $\GG$ be a finite cofiltration of $X$ by $M$ with index interval $[s,t]$.
Then $\width_R(X) = \width_R(M) - t$.
\end{enumerate}	
\end{prop}

\section{Quasi-projective dimension}
 
In this section, we introduce quasi-projective dimension as well as quasi-projective length for complexes and study their basic properties.


We first recall the original definitions of quasi-projective dimension and quasi-projective length for objects of an abelian category.

\begin{dfn}[{\cite[Definitions~3.1 and 5.1]{GJT}}]\label{defmod}
Let $\cA$ be an abelian category with enough projective objects, and let $M$ be a non-zero object of $\cA$.
\begin{enumerate}[\rm(1)]
\item 
A {\em finite quasi-projective resolution} of $M$ is a bounded complex $P$ of projective objects such that $\sH_i(P)\cong M^{\oplus a_i}$ for non-negative integers $a_i$, not all zero.

\item 
The {\em quasi-projective dimension} $\qpd_\cA(M)$ and the {\em quasi-projective length} $\qpl_\cA(M)$ of $M$ are defined by
\begin{align*}
\qpd_\cA(M)&:=\inf\{\sup(P) - \hsup(P) \mid P\text{ is a finite quasi-projective resolution of }M\},\\
\qpl_\cA(M)&:=\inf\{\sup(P) - \inf(P) \mid P\text{ is a finite quasi-projective resolution of }M\}.
\end{align*}
\end{enumerate}
By convention, we set $\qpd_\cA(0)=\qpl_\cA(0)=-\infty$.
\end{dfn}

\begin{rmk}\label{rmk:qpr-pd}
If a finite quasi-projective resolution $P$ of $M$ realizes $\qpd_\cA(M)$, that is $\qpd_\cA(M) = \sup(P) - \hsup(P)$, then $\pd_\cA(P)=\sup(P)$.
Indeed, otherwise one may replace $P$ by a projective resolution $Q$ with $\sup(Q) = \pd_\cA(P)<\sup(P)$.
Since $P$ and $Q$ are quasi-isomorphic, $Q$ is again a finite quasi-projective resolution of $M$ such that $\hsup(Q) = \hsup(P)$.
Therefore, $\qpd_\cA(M) \le \sup(Q) - \hsup(Q) < \sup(P) - \hsup(P) = \qpd_\cA(M)$, a contradiction.
\end{rmk}


The preceding definition uses only the homology of a complex of projective objects. 
For a general complex $M$, however, its homology does not determine $M$ in the derived category. 
To keep track of the additional information contained in $M$, we use filtrations by $M$.

\begin{dfn}\label{deffilt}
Let $\cA$ be an abelian category with enough projective objects, and let $M\in\Db(\cA)$ be a non-zero object.
\begin{enumerate}[\rm(1)]
\item 
A {\em finite quasi-projective filtration} by $M$ is defined to be a finite filtration of a non-zero object of $\cP(\cA)$ by $M$.

\item 
The {\em quasi-projective dimension} $\qpd_\cA(M)$ and the {\em quasi-projective length} $\qpl_\cA(M)$ of $M$ are defined by 
\begin{align*}
\qpd_\cA(M) &:= \inf\{\pd_\cA(F^s)-t\mid \FF\text{ is a finite quasi-projective filtration by }M\}, \\	
\qpl_\cA(M) &:= \inf\{\pd_\cA(F^s)-s\mid \FF\text{ is a finite quasi-projective filtration by }M\}.
\end{align*}
If $\qpd_{\mathcal A}(M)<\infty$, we also set $\c_\cA(M) := \qpl_\cA(M) - \qpd_\cA(M)$.
\end{enumerate}
By convention, we set $\qpd_\cA(0)=\qpl_\cA(0)=-\infty$.
\end{dfn}

For an object $M$ of $\cA$, we have defined its quasi-projective dimension and quasi-projective length in two ways, using the notation $\qpd_\cA(M)$ and $\qpl_\cA(M)$ in both cases.
This abuse of notation is justified by the following proposition.

\begin{prop}\label{modcpx}
Let $M$ be a nonzero object of $\cA$.
Then the quasi-projective dimensions and quasi-projective lengths of $M$ defined in Definitions~\ref{defmod} and \ref{deffilt} coincide.
\end{prop}

\begin{proof}
We write $\qpd_\cA^\triangle(M)$ and $\qpl_\cA^\triangle(M)$ for the invariants defined in Definitions~\ref{deffilt} and simply write $\qpd_\cA(M)$ and $\qpl_\cA(M)$ for the invariants defined in Definitions~\ref{defmod}.

Let $\FF=(F^{t+1}\to F^t\to\cdots\to F^{s+1}\to F^s)$ be a finite quasi-projective filtration by $M$, and write $\Cone(F^{i+1}\to F^i)\simeq M[i]^{\oplus a_i}$.
Backward induction on $i$ shows
\begin{align}\label{eq:modcpx}
\sH_k(F^i)\cong
\begin{cases}
M^{\oplus a_k}&\text{if }i\leq k\leq t,\\
0&\text{otherwise}.
\end{cases}
\end{align}
Indeed, the assertion is clear for $i=t$, since $F^t \simeq M[t]^{\oplus a_t}$.
Assume that $i < t$ and that (\ref{eq:modcpx}) holds for $i+1$.
From the exact triangle $F^{i+1} \to F^i \to M[i]^{\oplus a_i} \to F^{i+1}[1]$, we obtain an exact sequence
\[
0 \to \sH_{i}(F^{i+1}) \to \sH_{i}(F^{i}) \to M^{\oplus a_i} \to \sH_{i-1}(F^{i+1}) \to \sH_{i-1}(F^{i}) \to 0
\]
and isomorphisms $\sH_{k}(F^{i+1}) \cong \sH_{k}(F^{i})$ for all $k \neq i,i-1$.
Using the induction hypothesis, the above exact sequence yields an isomorphism $\sH_{i}(F^{i}) \cong M^{\oplus a_i}$ and $\sH_{i-1}(F^i) \cong 0$.
Therefore, (\ref{eq:modcpx}) holds for $i$.
Choose a projective resolution $P$ of $F^s$ such that $\sup(P) = \pd_\cA(F^s)$ and $\inf(P) = \hinf(F^s)$.
Since $P$ and $F^s$ are quasi-isomorphic, (\ref{eq:modcpx}) with $i=s$ shows that $P$ is a finite quasi-projective resolution of $M$.
Therefore, we get 
\begin{align*}
\qpd_\cA(M) &\le \sup(P) - \hsup(P) = \pd_\cA(F^s) - t \\
\qpl_\cA(M) &\le \sup(P) - \inf(P) = \pd_\cA(F^s) - s.
\end{align*}
Here, $t = \hsup(P)$ and $s = \hinf(F^s) = \inf(P)$ since $a_t, a_s > 0$.
Taking infima over $\FF$ gives $\qpd_\cA(M)\leq\qpd_\cA^\triangle(M)$ and $\qpl_\cA(M)\leq\qpl_\cA^\triangle(M)$.

Conversely, take a finite quasi-projective resolution $P$ of $M$.
Set $s:= \hinf(P)$ and $t := \hsup(P)$.
By taking good truncations, we obtain exact triangles
\[
\tau_{\ge i+1}P \to \tau_{\ge i}P \to \sH_i(P)[i] \to \tau_{\ge i+1}P[1]
\]
for $i = s,\ldots,t$.
Since $\sH_i(P) \cong M^{\oplus a_i}$ for some non-negative integers $a_i$, $\tau_{\ge t+1}P \simeq 0$, and $\tau_{\ge s}P \simeq P \in \cP(\cA)$, we get a finite quasi-projective filtration
\[
(\tau_{\ge t+1}P \to \tau_{\ge t}P \to \cdots \to \tau_{\ge s+1}P \to \tau_{\ge s}P)
\]
by $M$ in $\Db(\cA)$.
Consequently,
\begin{align*}
\qpd_\cA^\triangle(M) &\le \pd_\cA(P) -t \le \sup(P) - \hsup(P) \\
\qpl_\cA^\triangle(M) &\le \pd_\cA(P) -s \le \sup(P) - \hinf(P) \le \sup(P) - \inf(P).
\end{align*}
Taking infima over all finite quasi-projective resolutions $P$ of $M$ proves the reverse inequalities and hence we obtain $\qpd_\cA^\triangle(M) = \qpd_\cA(M)$ and $\qpl_\cA^\triangle(M) = \qpl_\cA(M)$.
\end{proof}

We next establish several basic properties of quasi-projective dimension for complexes that extend the corresponding results in the module case.

\begin{prop}\label{pd-qpd}
Let $\cA$ be an abelian category with enough projectives and let $M\in\Db(\cA)$.
Then
\[
\hsup(M) \le \qpd_\cA(M) \le \pd_\cA(M).
\]
If $\pd_\cA(M)<\infty$, then $\qpd_\cA(M)=\pd_\cA(M)$.
\end{prop}

\begin{proof}
First we note that if $\pd_\cA(M) < \infty$, then we can take $(F^1=0 \to F^0 = M)$ as a finite quasi-projective filtration by $M$ so that $\qpd_\cA(M) \le \pd_\cA(M)$.
Therefore, we may assume $M \not\simeq 0$ and $\qpd_\cA(M) < \infty$.
Take a finite quasi-projective filtration $\FF=(F^{t+1}\to F^t\to\cdots\to F^{s+1}\to F^s)$ by $M$.

First we prove $\hsup(M) \le \qpd_\cA(M)$.
It follows from \Cref{bound}(1)(ii) that $\hsup(M) = \hsup(F^s) - t$.
Since $\hsup(F^s) \le \pd_\cA(F^s)$, we get $\hsup(M) \le \pd_\cA(F^s) - t$.
Taking the infimum over all $\FF$ yields $\hsup(M) \le \qpd_\cA(M)$.

Next, assume $\pd_\cA(M) < \infty$ and prove $\qpd_\cA(M) \ge \pd_\cA(M)$ (the reverse inequality is proved above).
It follows from \Cref{thick} that every $F^i$ belongs to $\cP(\cA)$. 
Using the exact triangle $F^{i+1} \to F^i \to M[i]^{\oplus a_i} \to F^{i+1}[1]$, we get
\[
 \pd_{\cA}(F^{i+1})
 \le \max\{\pd_{\cA}(F^i), \pd_\cA(M) + i-1\} 
\]
for each $i= s,\ldots, t$.
By induction, we get $\pd_\cA(F^t) \le \max\{\pd_\cA(F^s), \pd_\cA(M) + t-2\}$; see \cite[Corollary 8.1.9]{CFH}.
On the other hand, as $F^t \simeq M[t]^{\oplus a_t}$ and $a_t > 0$, we have $\pd_\cA(F^t) = \pd_\cA(M) + t$.
This shows that $\pd_\cA(M) + t \le \pd_{\cA}(F^s)$ and so $\pd_\cA(M) \le \pd_\cA(F^s) -t$.
Taking the infimum over $\FF$ gives $\pd_\cA(M) \le \qpd_\cA(M)$.
Thus we are done.
\end{proof}

\begin{prop}\label{prop:basic-qpd}
Let $\cA$ be an abelian category with enough projective objects.
\begin{enumerate}[\rm(1)]
\item For $M\in\Db(\cA)$ and $n\in\ZZ$, one has $\qpd_\cA(M[n])=\qpd_\cA(M)+n$.
\item For $M\in\Db(\cA)$ and a positive integer $n$, one has $\qpd_\cA(M^{\oplus n})=\qpd_\cA(M)$.
\item For $M,N\in\Db(\cA)$, one has $\qpd_\cA(M\oplus N)\leq\max\{\qpd_\cA(M),\qpd_\cA(N)\}$.
\end{enumerate}
\end{prop}

\begin{proof}
(1) The assertion is clear if $M \simeq 0$ or $\qpd_\cA(M) = \infty$.
Assume $M \not\simeq 0$ and $\qpd_\cA(M) < \infty$.
Let $\FF = (F^{t+1} \to F^t \to \cdots \to F^{s+1} \to F^s)$ be a finite quasi-projective filtration by $M$ with $\qpd_\cA(M) = \pd_\cA(F^s)-t$.
Then $\FF[n] := (F^{t+1}[n] \to F^t[n] \to \cdots \to F^{s+1}[n] \to F^s[n])$ is a quasi-projective filtration by $M[n]$. 
Thus $\qpd_\cA(M[n]) \le \pd_\cA(F^s[n]) - t = \pd_\cA(F^s) - t + n = \qpd_\cA(M) + n$.
Using the same argument, $\qpd_{\cA}(M) = \qpd_\cA((M[n])[-n]) \le \qpd_\cA(M[n]) -n$.
Therefore, we obtain $\qpd_\cA(M[n]) = \qpd_\cA(M) + n$.

(2) The termwise direct sum of $n$ copies of a finite quasi-projective filtration by $M$ is a quasi-projective filtration by $M^{\oplus n}$, while every finite quasi-projective filtration by $M^{\oplus n}$ is also a finite quasi-projective filtration by $M$.
This shows the equality.

(3) Let
$\FF=(F^{t+1}\to F^t\to\cdots\to F^{s+1}\to F^s)$ and $\GG=(G^{v+1}\to G^v\to\cdots\to G^{u+1}\to G^u)$
be finite quasi-projective filtrations by $M$ and $N$, respectively.
Set $\Cone(F^{i+1} \to F^{i}) \simeq M[i]^{\oplus a_i}$ and $\Cone(G^{j+1} \to G^{j}) \simeq N[j]^{\oplus b_j}$ for each $i,j$.
Put $S=s+u$, $T=t+v$, and
\[
c_k=\sum_{i+j=k}a_ib_j \qquad(k = S, \ldots, T).
\]
For each $j \in [u,v]$, we extend the shifted filtration $\FF[j]$ to the common indexing interval $[S,T]$ by adding zero objects above degree $t+j$ and repeating $F^s[j]$ with identity maps below degree $s+j$ as follows
\[
\FF_j := (0 \to \cdots \to F^{t+1}[j] \to F^{t}[j] \to F^{t-1}[j] \to \cdots \to  F^{s+1}[j] \to F^{s}[j] \xrightarrow{\id} F^{s}[j] \xrightarrow{\id} \cdots \xrightarrow{\id} F^{s}[j]).
\]
Setting $\widetilde{\FF} = \bigoplus_{j=u}^v \FF_j^{\oplus b_j}$, we obtain a finite quasi-projective filtration by $M$ satisfying 
\[
\Cone(\widetilde{F}^{k+1} \to \widetilde{F}^k) \simeq M[k]^{\oplus c_k} \qquad(k = S, \ldots, T).
\]
Similarly, we construct a finite quasi-projective filtration $\widetilde{\GG}$ by $N$ satisfying  
\[
\Cone(\widetilde{G}^{k+1} \to \widetilde{G}^k) \simeq N[k]^{\oplus c_k} \qquad(k = S, \ldots, T).
\]
Hence $\widetilde{\FF} \oplus \widetilde{\GG}$ is a quasi-projective filtration by $M \oplus N$.
Moreover, as $\widetilde{F}^S = \bigoplus_{j=u}^v F^{s}[j]^{\oplus b_j}$, we have 
\[
\pd_\cA(\widetilde{F}^S) = \max\{\pd_\cA(F^s) + j \mid j = u, \ldots, v\} = \pd_\cA(F^s) + v.
\]
Similarly, $\pd_\cA(\widetilde{G}^S) = \pd_\cA(G^u) + t$.
Thus, 
\begin{align*}
\qpd_\cA(M \oplus N) \le \pd_\cA(\widetilde{F}^S \oplus \widetilde{G}^S) - T 
&= \max\{\pd_\cA(\widetilde{F}^S), \pd_\cA(\widetilde{G}^S)\} - T \\
&= \max\{\pd_\cA(F^s) + v, \pd_\cA(G^u) + t\} - (t+v)\\
&= \max\{\pd_\cA(F^s) -t, \pd_\cA(G^u) -v\}.
\end{align*}
Taking the infima over all finite quasi-projective filtrations $\FF$ and $\GG$, we obtain $\qpd_\cA(M \oplus N) \le \max\{\qpd_\cA(M), \qpd_\cA(N)\}$.
\end{proof}

For $\cA = \Mod R$, we simply write $\qpd_R = \qpd_{\Mod R}$, $\qpl_R = \qpl_{\Mod R}$, and $\c_R = \c_{\Mod R}$.
For homologically finite complexes, it does not matter whether the filtration is constructed in $\Mod R$ or in $\mod R$.

\begin{prop}
Let $M \in \Dfb(R)$ be a non-zero object.
If $\qpd_R(M) < \infty$, then there is a finite quasi-projective filtration by $M$ in $\Dfb(R)$.
Moreover, $\qpd_R(M) = \qpd_{\mod R}(M)$, $\qpl_R(M) = \qpl_{\mod R}(M)$, and $\c_R(M) = \c_{\mod R}(M)$ hold.  
\end{prop}

\begin{proof}
Let $\FF = (F^{t+1} \to F^{t} \to \cdots \to F^{s+1} \to F^s)$ be a finite quasi-projective filtration by $M$ in $\Db(R)$.
Applying \Cref{thick} for $\cX = \Dfb(R)$, each $F^i$ belongs to $\Dfb(R)$.
Hence every finite quasi-projective filtration by $M$ in $\Db(R)$ is already a filtration in $\Dfb(R)$.
This shows $\qpd_{\mod R}(M) \le \qpd_R(M)$ and $\qpl_{\mod R}(M) \le \qpl_R(M)$.
The reverse inequalities are easy.
\end{proof}

The following results are specific to commutative algebra.

\begin{prop}\label{proploc}
Let $M \in \Db(R)$.
\begin{enumerate}[\rm(1)]
\item	
For a finite quasi-projective filtration $\FF = (F^{t+1} \to F^{t} \to \cdots \to F^{s+1} \to F^s)$ by $M$, one has 
\[
\Supp_R(M) = \Supp_R(F^i) \qquad (i = s,\ldots,t).
\]

\item
For a prime ideal $\fp$, one has $\qpd_{R_\fp}(M_\fp) \le \qpd_{R}(M) \mbox{ and } \qpl_{R_\fp}(M_\fp) \le \qpl_{R}(M)$.

\item
For a non-zero $P \in \Pf(R)$, one has $\qpd_{R}(M \ltensor_R P) \le \qpd_R(M) + \pd_R(P) \mbox{ and } \qpl_{R}(M \ltensor_R P) \le \qpl_R(M) + \pd_R(P)$.
\end{enumerate}
\end{prop}

\begin{proof}
(1) The inclusion $\Supp_R(F^i) \subseteq \Supp_R(M)$ follows from \Cref{thick}, applied to $\cX = \{X \in \Db(R) \mid \Supp_R(X) \subseteq \Supp_R(M)\}$.
For the reverse inclusion, fix $\fp \in \Supp_R(M)$ and prove $(F^i)_\fp \not\simeq 0$.
Then $\FF_\fp = ((F^{t+1})_\fp \to (F^{t})_\fp \to \cdots \to (F^{s+1})_\fp \to (F^s)_\fp)$ is a finite filtration of $(F^s)_\fp$ by $M_\fp$.
Applying the same argument as in the proof of \Cref{bound}(1)(ii), we obtain $\hsup((F^i)_\fp) = \hsup(M_\fp) + t > -\infty$ and so $(F^i)_\fp \not\simeq 0$.

(2) We may assume $\fp \in \Supp_R(M)$, since otherwise the inequalities are trivial.
Let $\FF = (F^{t+1} \to F^{t} \to \cdots \to F^{s+1} \to F^s)$ be a quasi-projective filtration by $M$.
Then $\FF_\fp = ((F^{t+1})_\fp \to (F^{t})_\fp \to \cdots \to (F^{s+1})_\fp \to (F^s)_\fp)$ is a quasi-projective filtration by $M_\fp$ in $\Db(R_\fp)$.
Indeed, $(F^s)_\fp$ is a non-zero complex of finite projective dimension by (1).
It follows that $\qpd_{R_\fp}(M_\fp) \le \qpd_{R}(M)$ and $\qpl_{R_\fp}(M_\fp) \le \qpl_{R}(M)$.

(3) The proof is similar to that of (2).
We use the fact that $\pd_R(Q \ltensor_R P) \le \pd_R(Q) + \pd_R(P)$ for complexes $P,Q$ of finite projective dimension.
\end{proof}

Recall that an object $M \in \Dfb(R)$ is said to be {\em virtually small} if $M \simeq 0$ or the thick closure $\thick(M)$ of $M$ contains a non-zero perfect complex, where $\thick(M)$ denotes the smallest thick subcategory of $\Dfb(R)$ containing $M$.

\begin{prop}\label{vir}
Let $M \in \Dfb(R)$.
If $\qpd_R(M) < \infty$, then $M$ is virtually small.	
\end{prop}

\begin{proof}
We may assume $M \not\simeq 0$.
Let $\FF = (F^{t+1} \to F^t \to \cdots \to F^{s+1} \to F^s)$ be a finite quasi-projective filtration by $M$.
Applying \Cref{thick} to $\cX = \thick(M)$, we have $F^s \in \thick(M)$.
Since $F^s$ is a non-zero perfect complex, it follows that $M$ is virtually small.
\end{proof}

The following proposition extends \cite[Corollary 6.21]{GJT}; our argument differs from theirs.

\begin{prop}
Assume that $R$ is a noetherian local ring admitting a dualizing complex $D_R$.
Then $\qpd_R(D_R) < \infty$ if and only if $R$ is Gorenstein.	
\end{prop}

\begin{proof}
If $R$ is Gorenstein, then $D_R$ is isomorphic to a shift of $R$ and hence has finite projective dimension.
Thus, \Cref{pd-qpd} shows $\qpd_R(D_R) = \pd_R(D_R) < \infty$.
Conversely, if $\qpd_R(D_R) < \infty$, then $D_R$ is virtually small by \Cref{vir}.
Hence, $R$ is Gorenstein by \cite[3.3.5]{Gbook}. 
\end{proof}

Recently, Chen, Hu, and Yang introduced another quasi-projective dimension for complexes (\cite{CHY}).
We recall their definition.

\begin{dfn}[\cite{CHY}]
Let $M \in \Db(\cA)$ be a non-zero object.
A {\em finite quasi-projective resolution} of $M$ is a bounded complex $P$ of projective objects for which there exist integers $l \le u$ and non-negative integers $a_l,\ldots,a_u$, not all zero, satisfying
\[
\sH_*(P)
\cong
\bigoplus_{i=l}^u \sH_*(M[i]^{\oplus a_i})
\]
as graded objects in $\cA$.

We define the {\em quasi-projective dimension} of $M$ by 
\[
\qpd_\cA^{\mathsf{CHY}}(M) := \inf \{\sup(P) - \hsup(P) \mid \mbox{$P$ is a finite quasi-projective resolution of $M$}\}.
\]
By convention, we set $\qpd_\cA^{\mathsf{CHY}}(0) = -\infty$.
\end{dfn}

Our invariant agrees with theirs for modules, since both recover the invariant of \cite{GJT}, but the two invariants differ for complexes.
For example, our invariant changes under shifts, whereas $\qpd_R^{\mathsf{CHY}}$ is shift-invariant.
The following example also shows that the Auslander--Buchsbaum formula in \cite[Theorem~4.4]{CHY} does not hold as stated.

\begin{eg}\label{eg:homology}
Let $R=k[[x,y]]$ be the formal power series ring over a field $k$, and consider complexes
\[
P=(0\longrightarrow R^2\xrightarrow{(x\ \ y)}R\longrightarrow0),\qquad M=k\oplus R[1],
\]
where $R^2$ is in degree one.
A direct computation gives $\sH_0(P)\cong k$ and $\sH_1(P)\cong R$, so $\sH_*(P)\cong\sH_*(M)$ as graded modules.
Thus $\qpd_R^{\mathsf{CHY}}(M)\leq\sup(P)-\hsup(P)=0$.
By definition, $\qpd_R^{\mathsf{CHY}}(M)$ is always non-negative (except the case of $M \simeq 0$).
Therefore, $\qpd_R^{\mathsf{CHY}}(M) = 0$.
In addition,
\[
\depth_R(M)=\min\{\depth_R(k),\depth_R(R[1])\}=\min\{0,1\}=0,
\]
while $\hsup(M)=1$ and $\depth(R)=2$.
Consequently,
\[
\qpd_R^{\mathsf{CHY}}(M)+\hsup(M) = 1<2=\depth(R) - \depth_R(M),
\]
so the stated Auslander--Buchsbaum formula fails for this example.

On the other hand, since $M$ is perfect, $\qpd_R(M) = \pd_R(M) = 2$ by \Cref{pd-qpd}.
Thus, this complex also shows that $\qpd_R(M)$ and $\qpd_R^{\mathsf{CHY}}(M) + \hsup(M)$ need not agree.
\end{eg}

\section{Auslander--Buchsbaum, derived depth, and derived width formulas}

Throughout this section, $(R,\fm,k)$ denotes a commutative noetherian local ring.

This section develops the basic numerical consequences of finite quasi-projective dimension for complexes.
Using the filtration estimates from Section 2, we prove the Auslander-Buchsbaum formula, the derived depth formula, the dependency formula, and the derived width formula.

We first prove the Auslander--Buchsbaum formula for complexes with finite quasi-projective dimension, extending \cite[Theorem 4.4]{GJT} from modules to complexes.

\begin{thm}[Auslander--Buchsbaum formula]\label{AB}
Let $M\in\Dfb(R)$ and suppose that $\qpd_R(M)<\infty$.
Then
\[
\qpd_R(M)=\depth(R)-\depth_R(M).
\]
\end{thm}

\begin{proof}
The assertion is clear if $M\simeq0$.
Hence we may assume that $M\not\simeq0$.
Let $\FF=(F^{t+1}\to F^t\to\cdots\to F^{s+1}\to F^s)$ be a finite quasi-projective filtration by $M$ with $\qpd_R(M) = \pd_R(F^s) -t$.
\Cref{dwlem}(1) and the Auslander--Buchsbaum formula for perfect complexes (\cite[Corollary 16.4.2]{CFH}) give
\[
\qpd_R(M) = \pd_R(F^s)-t=\depth (R)-\depth_R(F^s)-t=\depth (R)-\depth_R(M).
\]
\end{proof}

For modules, the depth formulas for finite quasi-projective dimension have been established in several forms; see \cite[Theorem~4.11]{GJT}, \cite[Theorem~3.7]{JPMMR}, and \cite[Theorem~3.5]{FL}.
The following theorem simultaneously generalizes these results.

\begin{thm}[Derived depth formula]\label{DDF}
Let $M,N\in\Dfb(R)$, and suppose that $\qpd_R(M)<\infty$.
If $M\ltensor_RN\in\Dfb(R)$, then
\[
\depth_R(M\ltensor_RN)=\depth_R(M)+\depth_R(N)-\depth (R).
\]
\end{thm}

\begin{proof}
The assertion is clear if either $M$ or $N$ is zero.
Hence we may assume that both are non-zero.
Choose a finite quasi-projective filtration $\FF=(F^{t+1}\to F^t\to\cdots\to F^{s+1}\to F^s)$ by $M$.
Applying $-\ltensor_RN$ yields a finite filtration of $F^s\ltensor_RN$ by $M\ltensor_RN$ with the same index interval $[s,t]$.
Then \Cref{dwlem}(1) and the derived depth formula for perfect complexes (\cite[Theorem 16.4.3]{CFH}) give
\begin{align*}
\depth_R(M\ltensor_RN) = \depth_R(F^s\ltensor_RN) + t &= \depth_R(F^s) + \depth_R(N) - \depth(R) + t \\
&= \depth_R(M) + \depth_R(N) - \depth(R).
\end{align*}
\end{proof}

The following depth formula obtained in \cite{JPMMR} is a direct consequence of the derived depth formula; see the proof of \cite[Corollary 2.4]{Iye}.

\begin{cor}[Auslander's depth formula]
Let $M,N$ be non-zero finitely generated $R$-modules, and suppose that $\qpd_R(M)<\infty$.
Assume $M\ltensor_RN\in\Dfb(R)$ and set $q = \hsup(M \ltensor_R N)$.
If $q = 0$ or $\depth_R(\Tor_q^R(M,N)) \le 1$, then
\[
\depth_R(\Tor_q^R(M,N)) - q =\depth_R(M)+\depth_R(N)-\depth (R).
\]
\end{cor}

As in the module case, the derived depth formula yields Jorgensen's dependency formula below.
Its proof is the same as that of \cite[Theorem 3.6]{FL}.

\begin{thm}[Dependency formula]\label{dep}
Let $M,N\in\Dfb(R)$, and suppose that $\qpd_R(M)<\infty$ and $M\ltensor_RN\in\Dfb(R)$.
Then
\[
\hsup(M\ltensor_RN)=\sup\{\depth (R_\fp)-\depth_{R_\fp}(M_\fp) - \depth_{R_\fp}(N_\fp) \mid \fp\in\Supp_RM\cap\Supp_RN\}.
\]
\end{thm}

As a counterpart to the derived depth formula, we obtain the following derived width formula.

\begin{thm}[Derived width formula]\label{DWF}
Let $M,N\in\Dfb(R)$, and suppose that $\qpd_R(M)<\infty$.
If $\rHom_R(M,N)\in\Dfb(R)$, then
\[
\width_R(\rHom_R(M,N))=\depth_R(M) + \width_R(N) - \depth(R).
\]
\end{thm}

\begin{proof}
The assertion is clear if either $M$ or $N$ is zero.
Hence we may assume that both are non-zero.
Choose a finite quasi-projective filtration $\FF=(F^{t+1}\to F^t\to\cdots\to F^{s+1}\to F^s)$ by $M$.
Applying $\rHom_R(-,N)$ yields a finite cofiltration of $\rHom_R(F^s,N)$ by $\rHom_R(M,N)$ with the same index interval $[s,t]$.
Then \Cref{dwlem}(1),(2) and the derived width formula for perfect complexes (\cite[Theorem 16.4.4]{CFH}) yield
\begin{align*}
\width_R(\rHom_R(M,N)) = \width_R(\rHom_R(F^s,N)) + t &= \depth_R(F^s) + \width_R(N) - \depth(R) + t\\
&= \depth_R(M) + \width_R(N) - \depth(R).
\end{align*}
\end{proof}

Finally, we prove a width counterpart that encompasses both Auslander's depth formula and the dependency formula.

\begin{cor}
Let $M,N\in\Dfb(R)$, and suppose that $\qpd_R(M)<\infty$.
Assume $\rHom_R(M,N)\in\Dfb(R)$ and set $p := -\hinf(\rHom_R(M,N))$.
Then
\[
p = \depth(R) - \depth_R(M) - \width_R(N). 
\]
\end{cor}

\begin{proof}
It follows from \cite[Proposition 16.2.5]{CFH} that $\width_R(\rHom_R(M,N)) = \hinf(\rHom_R(M,N)) = -p$.
Therefore, the derived width formula (\Cref{DWF}) shows
\[
-p = \depth_R(M) + \width_R(N) - \depth(R).
\]
\end{proof}

\section{Quasi-projective dimension and complete intersections}

Throughout this section, $(R,\fm,k)$ denotes a commutative noetherian local ring.

As pointed out by Gheibi--Jorgensen--Takahashi (\cite{GJT}), modules of finite quasi-projective dimension are expected to behave similarly to modules over complete intersection local rings, or more generally, to modules of finite complete intersection dimension. 
For example, the derived depth formula is known for complexes of finite complete intersection dimension; see, for instance, \cite{CJ2} and the references therein.
Moreover, Gheibi--Jorgensen--Takahashi (\cite[Corollary 3.8]{GJT}) proved that over a complete intersection local ring that is a deformation of a regular local ring, every finitely generated module has finite quasi-projective dimension. 

In this section, we extend this picture to homologically finite complexes.
We first prove that over a complete intersection local ring with infinite residue field that is a deformation of a regular local ring, every object of $\Dfb(R)$ has finite quasi-projective dimension. 
More precisely, we prove the stronger result that the difference between quasi-projective length and quasi-projective dimension is bounded by the complexity. 
We then establish a new intersection theorem for complexes of finite quasi-projective dimension and use it, together with the derived depth formula, to obtain a result analogous to the theorem of Celikbas--Piepmeyer (\cite{CP}).

We begin by recalling the definitions of complete intersection dimension and complexity for complexes.

\begin{dfn}[{\cite{AGP,SW}}]
Let $M \in \Dfb(R)$.
\begin{enumerate}[\rm(1)]
\item	
A {\em quasi-deformation} of $R$ is a diagram of local homomorphisms $R \to R' \gets Q$ such that $R \to R'$ is flat and $Q \to R'$ is surjective with kernel generated by a $Q$-regular sequence.
The {\em complete intersection dimension} of $M$ is 
\[
\CIdim_R(M) := \inf\{\pd_Q(M \ltensor_R R') - \pd_Q(R') \mid \mbox{$R \to R' \gets Q$ is a quasi-deformation}\}.
\]
\item
The {\em complexity} of $M$ is defined by
\[
\cx_R(M) := \inf\{c \in \NN \mid \mbox{there exists $\alpha \in \RR$ such that $\dim_k(\Ext_R^n(M,k)) \le \alpha n^{c-1}$ for $n \gg 0$}\}.
\]
\end{enumerate}
\end{dfn}

\begin{rmk}
If $M$ has finite complete intersection dimension, then $M$ has finite complexity.
Moreover, $R$ is a complete intersection if and only if every $M\in\Dfb(R)$ has finite complete intersection dimension; see \cite[Proposition~3.5]{SW}.
We refer to \cite{AGP, SW} for details on complete intersection dimension.
\end{rmk}

We now prove the first main theorem of this section.

\begin{thm}\label{cxbound}
Let $R$ be a quotient of a noetherian local ring $Q$ with infinite residue field by a $Q$-regular sequence $x_1,\ldots,x_c$. 
For every non-zero object $M \in \Dfb(R)$ with $\pd_Q(M) < \infty$, one has $\qpd_R(M) < \infty$ and $\qpl_R(M) -\qpd_R(M) \le \cx_R(M)$.

In particular, if $Q$ is regular, then every non-zero object $M \in \Dfb(R)$ has finite quasi-projective dimension.
\end{thm}

\begin{proof}
By \cite[Theorem 5.9]{AGP}, $Q \twoheadrightarrow R$ factors as $Q \to Q' \to R$ such that the kernel of $Q' \to R$ is generated by a regular sequence, $\pd_{Q'}(M) < \infty$, and $\cx_R(M) = \pd_{Q'}(R)$.
Indeed, by taking a stupid truncation of a free resolution of $M$, we obtain an exact triangle $P \to M \to A[m] \to P[1]$ for some perfect complex $P$ and a finitely generated $R$-module $A$.
Since $\cx_R(M) = \cx_R (A)$, we can take such $Q'$ by applying the theorem to $A$.
Thus, we may replace $Q$ by $Q'$ and prove that $\qpd_R(M)< \infty$ and $\qpl_R(M) - \qpd_R(M) \le c$.

By replacing $M$ with an isomorphic complex, we may assume that $M$ is a bounded complex of finitely generated $R$-modules.
Take a left Cartan--Eilenberg resolution $P_{\mdot,\mdot}$ of $M$ over $Q$; see \cite[Section~5.7]{Wei}.
We may choose it so that $P_{p,\mdot}=0$ whenever $M_p=0$.
In particular, $P_{p,\mdot}=0$ for $|p|\gg0$.
Moreover, for each $p$, the complex $P_{p,\mdot}$ is a projective resolution of $M_p$, and $\Tot(P_{\mdot,\mdot})\simeq M$ in $\D(Q)$.

Set $E_{\mdot,\mdot}:=P_{\mdot,\mdot}\otimes_QR$.
Since $x_1,\ldots,x_c$ annihilate every $M_p$, the Koszul resolution of $R$ over $Q$ gives
\[
\sH_q^v(E_{p,\mdot})
\cong
\Tor_q^Q(R,M_p)
\cong
M_p^{\oplus\binom{c}{q}}
\]
for every $p$ and $q$.
By construction of the left Cartan--Eilenberg resolution, the horizontal maps induced on vertical homology are the differentials of $M$, and hence, as complexes, $\sH_q^v(E_{\mdot,\mdot}) \cong M^{\oplus\binom{c}{q}}$.

For $i\geq0$, let $\tau_{\geq i}^v(E_{\mdot,\mdot})$ denote the good truncation in the vertical direction and set $F^i := \Tot(\tau_{\geq i}^v(E_{\mdot,\mdot}))$. 
Here, the $(p,q)$ term is given by
\[
\tau_{\ge i}^v(E_{\mdot,\mdot})_{p,q} \cong 
\begin{cases}
E_{p,q} & (q > i) \\
\Ker(d^v_{p,q}) & (q=i) \\
0 & (q < i).	
\end{cases}
\] 
Consider the quotient double complex $C^{(i)} :=\tau_{\ge i}^v(E_{\mdot,\mdot})/\tau_{\ge i+1}^v(E_{\mdot,\mdot})$.
By the definition of the good truncation, $C^{(i)}$ is concentrated in vertical degrees $i$ and $i+1$, and these are
\[
C^{(i)}_{p,i}
=
\Ker(d^v_{p,i})
\qquad\text{and}\qquad
C^{(i)}_{p,i+1}
\cong
\Im(d^v_{p,i+1}).
\]
Moreover, the vertical differential $C^{(i)}_{p,i+1} \to C^{(i)}_{p,i}$ is identified with the natural inclusion $\Im(d^v_{p,i+1}) \hookrightarrow \Ker(d^v_{p,i})$.
It follows that
\[
\Tot(C^{(i)})
\cong
\Cone\left(
\Im(d^v_{\mdot,i+1})
\hookrightarrow
\Ker(d^v_{\mdot,i})
\right)[i].
\]
Therefore, we get
\begin{align*}
\Cone(F^{i+1} \to F^i) \simeq F^i/F^{i+1} \cong \Tot(C^{(i)}) 
&\cong 
\Cone\left(
\Im(d^v_{\mdot,i+1})
\hookrightarrow
\Ker(d^v_{\mdot,i})
\right)[i] \\
&\simeq \sH_i^v(E_{\mdot,\mdot})[i] 
\cong M[i]^{\oplus \binom{c}{i}}.	
\end{align*}

Next, we will prove that $(F^{c+1} \to F^c \to \cdots \to F^1 \to F^0)$ is a finite quasi-projective filtration by $M$.
As $P_{\mdot,q} \cong 0$ for $q < 0$, $\tau_{\ge 0}^v(E_{\mdot,\mdot}) \cong E_{\mdot,\mdot} $ and so $F^0 \cong \Tot(E_{\mdot,\mdot} ) \cong \Tot(P_{\mdot,\mdot}) \otimes_Q R \simeq M \ltensor_Q R \in \Pf(R)$.
Since
$\sH_q^v(E_{\mdot,\mdot})
\cong M^{\oplus\binom{c}{q}}
\cong 0
\,\, (q>c)$, 
one has
$\sH_q^v(\tau_{\ge c+1}^v(E_{\mdot,\mdot})) \cong 0$ for all $q$. 
Hence the spectral sequence
\[
E^1_{p,q}
=
\sH_q^v(\tau_{\ge c+1}^v(E_{p,\mdot}))
\Longrightarrow
\sH_{p+q}(F^{c+1})
\]
shows that $F^{c+1}\simeq0$.
Therefore, $(F^{c+1} \to F^c \to \cdots \to F^1 \to F^0)$ is a finite quasi-projective filtration by $M$.
Thus $M$ has finite quasi-projective dimension and moreover $\qpl_R(M) \le \pd_R(F^0) = \pd_R(M \ltensor_Q R)\le \pd_Q(M)$.
In addition, the Auslander--Buchsbaum formula (\Cref{AB}) shows 
\[
\pd_Q(M) = \depth(Q) - \depth_Q(M) = \depth(R) + c - \depth_R(M) = \qpd_R(M) +c,
\]
where the second equality is due to $\depth(R) = \depth(Q) -c$ and \cite[Proposition 5.2(1)]{Iye}.
Consequently, we obtain $\c_R(M) = \qpl_R(M) - \qpd_R(M) \le \pd_Q(M) - (\pd_Q(M) - c) = c$.
\end{proof}

Therefore, over such a complete intersection, every non-zero object of $\Dfb(R)$ has finite quasi-projective dimension.
A converse in the following sense also holds.

\begin{prop}
If every non-zero object	 $M \in \Dfb(R)$ has finite quasi-projective dimension, then $R$ is a complete intersection.
\end{prop}

\begin{proof}
As we have proved in \Cref{vir}, every object of $\Dfb(R)$ is virtually small.
Thus the result follows by \cite[Theorem 5.2]{Pol}.
\end{proof}

\subsection{The new intersection theorem}

Sharif and Yassemi proved the following extension of the new intersection theorem.
Their formulation uses the quasi-projective dimension $\qpd_R^{\mathsf{AGP}}(M)$ of Avramov–Gasharov–Peeva, which equals $\CIdim_R(M) + \cx_R(M)$ when $M$ has finite complete intersection dimension; see \cite[Theorem 5.11]{AGP}\footnote{Please be careful not to confuse their quasi-projective dimension with quasi-projective dimension in our sense. In fact, we expect that $\qpd_R^{\mathsf{AGP}}(M) = \qpl_R(M)$ whenever $\CIdim_R(M) < \infty$ as formulated in the question below.}

\begin{thm}[{\cite[Theorem~3.3]{SY}}]\label{thm:SY}
Let $M,N \in \Dfb(R)$ be non-zero objects.
If $\CIdim_R(M) < \infty$, then the inequality
\[
\dim_R(N) \le \dim_R(M\ltensor_R N) + \qpd_R^{\mathsf{AGP}}(M)
\]	
holds.
\end{thm}

We now prove the corresponding statement with quasi-projective length as the correction term.

\begin{thm}[New intersection theorem]\label{NIT}
Let $M,N\in\Dfb(R)$ be non-zero objects.
If $\qpd_R(M) < \infty$, then the inequality 
\[
\dim_R(N) \le \dim_R(M\ltensor_R N)+\qpl_R(M)
\] 
holds.
\end{thm}

\begin{proof}
Take a finite quasi-projective filtration $\FF := (F^{t+1} \to F^t \to \cdots \to F^{s+1} \to F^s)$ by $M$ with $\qpl_R(M) = \pd_R(F^s) -s$.
By \Cref{proploc}(1), the complexes $F^s \ltensor_R N$ and $M\ltensor_R N$ have the same support.
Since $\FF \ltensor_R N$ is a finite filtration of $F^s \ltensor_R N$ by $M \ltensor_R N$, \cite[Corollary A.4.16]{Gbook} and \Cref{bound}(1) give
\begin{align*}
 \hinf(F^s_\fp\ltensor_{R_\fp}N_\fp)
 &=s + \hinf (M_\fp) + \hinf (N_\fp)\\
 &=s + \hinf(M_\fp\ltensor_{R_\fp}N_\fp)
\end{align*}
for each $\fp \in \Supp_R(M\ltensor_R N)$.
It follows from the definition of dimension that $\dim_R(F^s \ltensor_R N) = \dim_R(M\ltensor_R N) - s$.
Hence, the new intersection theorem for perfect complexes (\cite[Corollary 18.5.5]{CFH}) yields
\begin{align*}
\dim_R(N) &\le \dim_R(F^s \ltensor_R N) + \pd_R(F^s) \\
&= \dim_R(M\ltensor_R N) - s + \pd_R(F^s) \\
&= \dim_R(M\ltensor_R N) + \qpl_R(M).
\end{align*}
\end{proof}

\subsection{Serre's conditions for derived tensor products}

Celikbas and Piepmeyer (\cite{CP}) proved a descent theorem for Serre’s conditions under a finite complete intersection dimension hypothesis.
We recall it in the form relevant here.

\begin{thm}[{\cite[Theorem 3.1]{CP}}]\label{thm:CP}
Let $M,N\in\Dfb(R)$ be non-zero objects, let $n$ be a non-negative integer, and assume that $\CIdim_R(M)<\infty$. 
Put $c=\cx_R(M)$.
Assume that $M \ltensor_R N \in \Dfb(R)$ and that $M \ltensor_R N$ satisfies $(S_{n+c})$.	
Then $N$ satisfies $(S_n)$.
\end{thm}

To prove the version of their theorem for finite quasi-projective dimension, the following lemma plays a key role.

\begin{lem}\label{key}
Let $M\in\Dfb(R)$ be a non-zero object with $\qpd_R(M) < \infty$.
\begin{enumerate}[\rm(1)]
\item	
If $\FF$ is a finite quasi-projective filtration by $M$ with index interval $[s,t]$, then $\qpd_R(M) = \pd_R(F^s) -t$.
\item
The equality
\[
\c_R(M) = \inf\{\range(\FF) \mid \FF \mbox{ is a finite quasi-projective filtration by $M$}\}
\]
holds.
\item
$\c_R(M) = 0$ if and only if $M$ is perfect.
\end{enumerate}
\end{lem}

\begin{proof}
(1) By the Auslander--Buchsbaum formula (\Cref{AB}) and Lemma~\ref{dwlem}(1), we get
\[
\qpd_R(M) = \depth(R) - \depth_R(M) = \depth(R) - \depth_R(F^s) - t = \pd_R(F^s) -t.
\]

(2) Let $\FF$ be a finite quasi-projective filtration by $M$ with index interval $[s,t]$.
Then (1) shows $\c_R(M) = \qpl_R(M) - \qpd_R(M) \le (\pd_R(F^s) -s) - (\pd_R(F^s) -t) = t-s = \range(\FF)$.
Moreover, if we choose $\FF$ so that $\qpl_R(M) = \pd_R(F^s) - s$, then $\c_R(M) = \qpl_R(M) - \qpd_R(M) = (\pd_R(F^s) -s) - (\pd_R(F^s) -t) = t-s = \range(\FF)$ holds again using (1).
This shows the equality
\[
\c_R(M) = \inf\{\range(\FF) \mid \FF \mbox{ is a finite quasi-projective filtration by $M$}\}.
\]

(3) If $M$ is perfect, then the finite quasi-projective filtration $(F^1=0 \to F^0 = M)$ by $M$ has range $0$.
Thus, (2) shows $\c_R(M) = 0$.
Conversely, if $\c_R(M) = 0$, then there is a finite quasi-projective filtration of the form $\FF =  (F^{s+1} \to F^s)$ by (2).
This implies that $F^s \simeq M[s]^{\oplus a_s}$ for some $a_s > 0$.
Since $F^s$ is perfect, so is $M$.
\end{proof}

\begin{prop}\label{qcx}
Let $M\in\Dfb(R)$ be a non-zero object with $\qpd_R(M) < \infty$.
Then $\c_{R_\fp}(M_\fp) \le \c_R(M)$ holds for any $\fp \in \Supp_R(M)$.
\end{prop}

\begin{proof}
Choose a finite quasi-projective filtration by $M$ of $\range(\FF) = \c_R(M)$, which exists by \Cref{key}(2).
Localizing at $\fp$ gives a finite quasi-projective filtration $\FF_\fp$ by $M_\fp$ with the same index interval; its terms are non-zero by \Cref{proploc}(1). 
Hence, again using \Cref{key}(2), we obtain $\c_{R_\fp}(M_\fp) \le \range(\FF_\fp) = \range(\FF) = \c_{R}(M)$.
\end{proof}

Now, we are ready to prove the following theorem.

\begin{thm}\label{Serre}
Let $M,N\in\Dfb(R)$ be non-zero objects, let $n$ be a non-negative integer, and assume that $\qpd_R(M)<\infty$. 
Put $c=\c_R(M)$.
Assume that $M \ltensor_R N \in \Dfb(R)$ and that $M \ltensor_R N$ satisfies $(S_{n+c})$.	
Then $N$ satisfies $(S_n)$.
\end{thm}

\begin{proof}
First we note that $\qpd_{R_\fp}(M_\fp)<\infty$ holds for every $\fp \in \Spec(R)$ by \Cref{proploc}(2).

Fix $\fp \in \Supp_R(N)$ and prove
\[
\depth_{R_\fp}(N_\fp) + \hinf(N_\fp) \ge \min\{n,\height\fp\}.
\]
We divide the proof into the following three cases:

\emph{Case 1: $\fp \in \Supp_R(M) \cap \Supp_R(N)$.}

\emph{Case 2: $\fp \in \Supp_R(N) \setminus \Supp_R(M)$ and $\dim(R) \le n + c$.}

\emph{Case 3: $\fp \in \Supp_R(N) \setminus \Supp_R(M)$ and $\dim(R) > n + c$.}

\smallskip
\noindent\emph{Case 1: $\fp \in \Supp_R(M) \cap \Supp_R(N)$.}
\smallskip

In this case, Theorems~\ref{AB}, \ref{DDF} and \cite[Corollary A.4.16]{Gbook} give
\begin{align*}
 \depth_{R_\fp}(N_\fp) + \hinf(N_\fp) 
 &= (\depth_{R_\fp}(M_\fp \ltensor_{R_\fp} N_\fp) - \depth_{R_\fp}(M_\fp) + \depth(R_\fp)) + \hinf(N_\fp) \\
 &= \depth_{R_\fp}(M_\fp \ltensor_{R_\fp} N_\fp) + \hinf(M_\fp \ltensor_{R_\fp} N_\fp) + \qpd_{R_\fp}(M_\fp)-\hinf(M_\fp)\\
 &\ge \depth_{R_\fp}(M_\fp \ltensor_{R_\fp} N_\fp) + \hinf(M_\fp \ltensor_{R_\fp} N_\fp) \\
 &\ge \min\{n+c,\height \fp\}.
\end{align*}
Here, the first inequality follows from \Cref{pd-qpd}, while the last inequality follows from the assumption that $M\ltensor_RN$ satisfies $(S_{n+c})$.

\smallskip
\noindent\emph{Case 2: $\fp \in \Supp_R(N) \setminus \Supp_R(M)$ and $\dim(R) \le n + c$.}
\smallskip

Since $\fm \in \Supp_R(M) \cap \Supp_R(N)$, it follows from Case 1 and \cite[Lemma A.6.2]{Gbook} that
\begin{align*}
\depth_{R_\fp}(N_\fp) + \hinf(N_\fp) 
&\ge \depth_R(N) - \dim(R/\fp R) + \hinf(N) \\
&\ge \min\{n+c, \dim(R)\} - (\dim(R) - \height \fp) \\
&= \dim(R) - (\dim(R) - \height \fp) = \height \fp. 
\end{align*}

\smallskip
\noindent\emph{Case 3: $\fp \in \Supp_R(N) \setminus \Supp_R(M)$ and $\dim(R) > n + c$.}
\smallskip

Let $I = \ann_R \bigl(\bigoplus_i \sH_i(M)\bigr)$, so $\Supp_R(M)=V(I)$, and choose a minimal prime ideal $\fq$ of $I + \fp$.
Then $\fq\in\Supp_R M\cap\Supp_R N$.
If $\height \fq \le n+c$, then $\dim(R_\fq) \le n+c$. 
Applying Case 2 over $R_\fq$ to the prime $\fp R_\fq$ yields $\depth_{R_\fp}(N_\fp) + \hinf(N_\fp) \ge \height \fp$.
Assume $\height \fq > n+c$.
Then \Cref{NIT} yields
\begin{align}\label{eq:Serre}
\dim_{R_\fq}(R_\fq/\fp R_\fq) \le \qpl_{R_\fq}(M_\fq) + \dim_{R_\fq}(M_\fq \ltensor_{R_\fq} R_\fq/\fp R_\fq).
\end{align}
Since $\fq$ is a minimal prime ideal of $I + \fp$, one has $\Supp_{R_\fq}(M_\fq \ltensor_{R_\fq} R_\fq/\fp R_\fq)  = \{\fq R_\fq\}$.
Hence the definition of the dimension shows 
\begin{align}\label{eq:Serre2}
\dim_{R_\fq}(M_\fq \ltensor_{R_\fq} R_\fq/\fp R_\fq) = - \hinf(M_\fq \ltensor_{R_\fq} R_\fq/\fp R_\fq).	
\end{align}
Thus, we get 
\begin{align}\label{eq:Serre3}
\begin{split}
\depth_{R_\fp}(N_\fp)+\hinf(N_\fp)
&\geq \depth_{R_\fq}(N_\fq)+\hinf(N_\fq)
-\dim(R_\fq/\fp R_\fq)\\
&\geq \depth_{R_\fq}(N_\fq)+\hinf(N_\fq) -\qpl_{R_\fq}(M_\fq)\\
&\qquad\qquad\qquad\qquad
-\dim_{R_\fq}\bigl(M_\fq\ltensor_{R_\fq}R_\fq/\fp R_\fq\bigr)\\
&=\depth_{R_\fq}(N_\fq)+\hinf(N_\fq)
-\qpl_{R_\fq}(M_\fq)+\hinf(M_\fq).
\end{split}
\end{align}
where the first inequality follows from \cite[Lemma A.6.2]{Gbook}, the second one is due to (\ref{eq:Serre}), while the last one uses (\ref{eq:Serre2}).

On the other hand, since $\c_{R_\fq}(M_\fq) \le \c_{R}(M) = c$ by \Cref{qcx}, Theorems~\ref{AB} and \ref{DDF} imply
\begin{align}\label{eq:Serre4}
\begin{split}
\qpl_{R_\fq}(M_\fq) \le c + \qpd_{R_\fq}(M_\fq) &= c + \depth(R_\fq) - \depth_{R_\fq}(M_\fq) \\
&= c + \depth_{R_\fq}(N_\fq) - \depth_{R_\fq}(M_\fq \ltensor_{R_\fq} N_\fq).
\end{split}
\end{align}
Combining (\ref{eq:Serre3}) and (\ref{eq:Serre4}), we obtain
\begin{align*}
\depth_{R_\fp}(N_\fp) + \hinf(N_\fp) &\ge \depth_{R_\fq}(N_\fq) + \hinf(N_\fq) - (c + \depth_{R_\fq}(N_\fq) - \depth_{R_\fq}(M_\fq \ltensor_{R_\fq} N_\fq)) + \hinf(M_\fq) \\
&\ge -c + \depth_{R_\fq}(M_\fq \ltensor_{R_\fq} N_\fq) + \hinf(M_\fq \ltensor_{R_\fq} N_\fq) \\
&\ge -c + \min\{n+c, \height \fq\} = - c + (n+c) = n.  
\end{align*}
This completes the proof.
\end{proof}

Comparing Theorems~\ref{thm:SY}, \ref{thm:CP} with Theorems~\ref{NIT}, \ref{Serre}, and taking \Cref{cxbound} into account, the following question naturally arises.

\begin{ques}
Let $M \in \Dfb(R)$ be a non-zero object with $\CIdim_R(M) < \infty$.
Does $M$ have finite quasi-projective dimension?
If so, does the equality $\c_R(M) = \cx_R(M)$ hold?	
\end{ques}
\noindent
The first part of the question has been recently answered affirmatively for finitely generated modules by Dey, Ferraro, and Gheibi \cite{DFG}. 
However, the second part remains open even when $R$ is a deformation of a regular local ring.
\section{Vanishing of Ext and Tor}
In this section, $R$ is again an arbitrary commutative noetherian ring, unless otherwise specified.

The results of the preceding sections describe the numerical consequences of finite quasi-projective dimension.
We now turn to vanishing phenomena for Tor and Ext modules, as well as Tate (co)homology.
We then apply these vanishing results to the symmetry of Ext vanishing over Gorenstein rings, which is known to hold for complexes of finite complete intersection dimension.

We begin with Auslander type conditions for complexes with finite quasi-projective dimension.

\begin{prop}\label{AC}
Let $M,N\in\Dfb(R)$ be non-zero objects, and assume that $\qpd_R(M)<\infty$.
\begin{enumerate}[\rm(1)]
\item 
If $\hsup(M\ltensor_RN)<\infty$, then $\hsup(M\ltensor_RN) \le \qpd_R(M)+\hsup(N)$.
\item 
If $\hinf(\rHom_R(M,N))>-\infty$, then $\hinf(\rHom_R(M,N)) \ge \hinf (N)-\qpd_R(M)$.
\end{enumerate}
\end{prop}

\begin{proof}
Choose a finite quasi-projective filtration
\[
\FF=(F^{t+1}\to F^t\to\cdots\to F^{s+1}\to F^s)
\]
by $M$ such that $\qpd_R(M) = \pd_R(F^s)-t$.

(1) By \Cref{bound}(1)(ii) and \cite[A5.5]{Gbook}, one has
\[
\hsup(M\ltensor_R N) = \hsup(F^s \ltensor_R N) - t \le \pd_R(F^s) + \hsup(N) -t = \qpd_R(M) + \hsup(N).
\]

(2) By \Cref{bound2}(1)(i) and \cite[A5.3]{Gbook}, one has
\[
\hinf(\rHom_R(M,N)) = \hinf(\rHom_R(F^s, N)) + t \ge \hinf(N) - \pd_R(F^s) + t = \hinf(N) - \qpd_R(M).
\]
\end{proof}


We next study vanishing of Tate homology and Tate cohomology. 
Unlike ordinary Tor and Ext, Tate (co)homology is defined for complexes of finite Gorenstein dimension. 
We first recall complete resolutions and Gorenstein dimension.

\begin{dfn}
Let $M \in \Dfb(R)$.
\begin{enumerate}[\rm(1)]
\item
An acyclic complex $T$ of finitely generated projective $R$-modules is called {\em totally acyclic} if the complex $\Hom_R(T,R)$ is acyclic.
\item
A {\em complete resolution} of $M$ is a diagram $T \xrightarrow{\tau} P \xrightarrow{\pi} M$, where $\pi$ is a projective resolution of $M$, $T$ is totally acyclic, and $\tau_i$ is an isomorphism for $i \gg 0$.
\item
We define the {\em Gorenstein dimension} of $M$ by 
\[
\Gdim_R(M)
:=
\inf\left\{
n \ \middle|\
\begin{array}{l}
\mbox{there is a complete resolution } T \xrightarrow{\tau} P \xrightarrow{\pi} M \mbox{ such that}\\
\mbox{$\tau_i$ is an isomorphism for all $i \ge n$}
\end{array}
\right\}.
\]
Therefore, $\Gdim_R(M) < \infty$ if and only if $M$ admits a complete resolution.

Denote by $\cG(R)$ the full subcategory of $\Dfb(R)$ consisting of objects $M$ with $\Gdim_R(M)< \infty$.
\end{enumerate}
\end{dfn}

\begin{rmk}
It is known that $R$ is Gorenstein if and only if every $M \in \Dfb(R)$ has finite Gorenstein dimension; see \cite[2.3.14]{Gbook}.
Moreover, $\cG(R)$ is a thick subcategory of $\Dfb(R)$ containing $\Pf(R)$; see \cite[Lemma 2.1.12]{Gbook}.
\end{rmk}

If a complex has finite Gorenstein dimension, we can define the Tate (co)homology modules.

\begin{dfn}
Let $M,N$ be bounded complexes of finitely generated $R$-modules.
Assume that $\Gdim_R(M)< \infty$ and let $T \to P \to M$ be a complete resolution.

\begin{enumerate}[\rm(1)]
\item
Define the {\em $n$-th Tate homology module} of $M,N$ by 
\[
\widehat{\Tor}_n^R(M,N) := \sH_n(T \otimes_R N).  
\]	

\item
Define the {\em $n$-th Tate cohomology module} of $M,N$ by 
\[
\widehat{\Ext}^n_R(M,N) := \sH_{-n}(\Hom_R(T, N)).  
\]	
\end{enumerate}
These definitions are independent of the choices of complete resolutions.
\end{dfn}

\begin{prop}\label{tate-functor}
Tate homology and Tate cohomology induce bifunctors
\begin{align*}
\widehat{\Tor}^R_n(-,-)&: \cG(R)\times\Dfb(R)\longrightarrow \Mod(R) \\
\widehat{\Ext}_R^n(-,-)&: \cG(R)^\op\times\Dfb(R)\longrightarrow \Mod(R).
\end{align*}
Moreover, the following statements hold.
\begin{enumerate}[\rm(1)]
\item
For every $N\in\Dfb(R)$, the functor $\widehat{\Tor}^R_*(-,N)$ is homological. 

\item
For every $M\in\cG(R)$, the functor $\widehat{\Tor}^R_*(M,-)$ is homological.

\item
For every $N\in\Dfb(R)$, the functor $\widehat{\Ext}_R^*(-,N)$ is cohomological.

\item
For every $M\in\cG(R)$, the functor $\widehat{\Ext}_R^*(M,-)$ is cohomological.
\end{enumerate}
\end{prop}

\begin{proof}
By \cite[Propositions 4.6, 4.8]{Vel} and \cite[Propositions 2.8, 2.9]{CJ}, Tate cohomology and Tate homology are functorial on the corresponding categories of complexes and yield the required long exact sequences.
It remains to show that these functors factor through the corresponding derived categories.

For this, it suffices to show that these functors vanish whenever either variable is an acyclic complex, since the cone of a quasi-isomorphism between bounded complexes is bounded and acyclic.
The required vanishing therefore follows from \cite[Theorem 4.5]{Vel} for Tate cohomology and from \cite[Proposition 2.5 and Lemma 2.7]{CJ} for Tate homology.
Thus the Tate functors invert quasi-isomorphisms, and the long exact sequences descend to the stated derived categories.
\end{proof}

\begin{lem}\label{lemst}
Let $\cX \subseteq \Dfb(R)$ be a thick subcategory containing $\Pf(R)$, and let $H_* = (H_i)_{i \in \ZZ}$ be a homological functor from $\cX$ to an abelian category $\cA$ such that $H_*(P) \cong 0$ for all $P \in \Pf(R)$. 
Let $M \in \cX$ with $\qpd_R(M) < \infty$.
If $\inf(H_*(M)) > -\infty$ or $\sup(H_*(M)) < \infty$, then $H_*(M) \cong 0$.
Moreover, the same result is true for a cohomological functor.

The analogous statement also holds for a homological or cohomological functor on $\mathcal X^{\mathrm{op}}$.
\end{lem}

\begin{proof}
The assertion is clear if $M\simeq0$.
Hence we may assume that $M\not\simeq0$.
We only deal with the case where $\inf(H_*(M)) > -\infty$ because the argument is similar.
Let $\FF = (F^{t+1} \to F^t \to \cdots \to F^{s+1} \to F^s)$ be a finite quasi-projective filtration by $M$.
Note that as $F^s$ belongs to $\Pf(R)$, $H_*(F^s) \cong 0$, i.e., $\inf(H_*(F^s)) = \infty$.
From \Cref{bound}(1)(i), one has $\inf(H_*(M)) = \inf(H_*(F^s)) - s = \infty$.
Therefore, $H_*(M) \cong 0$.

For a functor on $\mathcal X^{\mathrm{op}}$, view $\FF$ as a finite cofiltration in $\mathcal X^{\mathrm{op}}$ and apply \Cref{bound2}.
\end{proof}

\begin{prop}\label{tatevan}
Let $M,N\in\Dfb(R)$.  
Assume that $M$ and $N$ have finite Gorenstein dimension, and that either $M$ or $N$ has finite quasi-projective dimension.
Then the following equivalences hold.
\begin{enumerate}[\rm(1)]
\item	
$\mathrm{(i)}\,\, \widehat{\Tor}_{\gg0}^R(M,N)=0
\Longleftrightarrow
\mathrm{(ii)}\,\,\widehat{\Tor}_{\ll0}^R(M,N)=0
\Longleftrightarrow
\mathrm{(iii)}\,\, \widehat{\Tor}_*^R(M,N)=0$

\item
$\mathrm{(i)}\,\, \widehat{\Ext}_R^{\gg0}(M,N)=0
\Longleftrightarrow
\mathrm{(ii)}\,\, \widehat{\Ext}_R^{\ll0}(M,N)=0
\Longleftrightarrow
\mathrm{(iii)}\,\, \widehat{\Ext}_R^*(M,N)=0$
\end{enumerate}
\end{prop}

\begin{proof}
The implications (iii) $\Rightarrow$ (i) and (iii) $\Rightarrow$ (ii) are trivial.
The opposite implications follow from \Cref{tate-functor} and \Cref{lemst} together with  \cite[Theorem 4.5]{Vel} and \cite[Proposition 2.5 and Lemma 2.7]{CJ}.
\end{proof}

Finally, we study the symmetry of eventual vanishing of Ext modules.
To prove our final theorem, we recall the following notion.

\begin{dfn}[\cite{Levins}]
Let $M \in \Dfb(R)$. 
The {\em Auslander bound of $M$} is defined by
\[
B_M^R := \sup \{\hinf(N) - \hinf(\rHom_R(M,N)) \mid 0 \not\simeq N \in \Dfb(R) \mbox{ such that } \hinf(\rHom_R(M,N)) > -\infty\}.
\]	
\end{dfn}

\begin{lem}\label{Ausbound}
Let $M \in \Dfb(R)$ with $\qpd_R(M) < \infty$.
Then $B_M^R \le \qpd_R(M)$. 
Moreover, if $\rHom_R(M,R) \in \Dfb(R)$, then $B_{\rHom_R(M,R)}^R \le - \hinf(M)$.
\end{lem}

\begin{proof}
If $M\simeq 0$, both assertions are immediate.
Assume that $M\not\simeq 0$.

The first assertion follows immediately from \Cref{AC}(2).

Assume $\rHom_R(M,R) \in \Dfb(R)$ and prove $B_{\rHom_R(M,R)}^R \le - \hinf(M)$.
Choose a finite quasi-projective filtration $\FF = (F^{t+1} \to \cdots F^{s+1} \to F^s)$ by $M$.
Let $N \in \Dfb(R)$ such that $\hinf(\rHom_R(M^*,N)) > -\infty$.
Then 
\[
\rHom_R(\FF^*,N) := (\rHom_R((F^{t+1})^*,N) \to \rHom_R((F^{t})^*,N) \to \cdots \to \rHom_R((F^{s})^*, N))
\]
is a finite filtration of $\rHom_R((F^{s})^*, N)$ by $\rHom_R(M^*,N)$.
Here, $(-)^* = \rHom_R(-, R)$.
Therefore, we get
\begin{align*}
\hinf(\rHom_R(M^*,N)) = \hinf(\rHom_R((F^s)^*,N)) - s &\ge \hinf(N) - \pd_R((F^s)^*) - s \\
&\ge \hinf(N) + \hinf(F^s) - s \\
&= \hinf(N) + \hinf(M).
\end{align*}
Here, the equalities follow from \Cref{bound}(1)(i) and the inequalities follow from \cite[A.5.3]{Gbook}.
This shows $B_{M^*}^R \le - \hinf(M)$ as desired.
\end{proof} 

The following theorem generalizes \cite[Corollary 5.4]{AMMRSL}.

\begin{thm}\label{symm0}
Assume $R$ admits a dualizing complex $D_R$ (for example, if $R$ is complete local).
Let $M, N \in \Dfb(R)$ such that $\qpd_R(M) < \infty$ and $\Gdim_R(M) < \infty$.
Then 
\[
\Ext_R^{\gg 0}(M,N) \cong 0 \Longleftrightarrow \Ext_R^{\gg 0}(N, M \ltensor_R D_R) \cong 0.
\]
\end{thm}

\begin{proof}
We may assume that $M$ and $N$ are non-zero.
Set $M^*:=\rHom_R(M,R)$.
Since $\Gdim_R(M)<\infty$, one has
$M^*\in\Dfb(R)$ and $\Gdim_R(M^*)<\infty$;
see \cite[(2.1.10) and (2.3.8)]{Gbook}.
By \Cref{Ausbound}, one also has
$
B_M^R<\infty
$ and 
$B_{M^*}^R<\infty$.
The assertion now follows from
\cite[Definition 3.1 and Theorem 4.4]{AMMRSL}.
\end{proof}

As an immediate consequence, we obtain symmetry of eventual Ext vanishing over Gorenstein local rings.

\begin{cor}\label{symm}
Suppose that $R$ is a Gorenstein local ring.
Let $M,N\in\Dfb(R)$ and assume that either $M$ or $N$ has finite quasi-projective dimension.
Then $\Ext_R^{\gg0}(M,N)=0$ if and only if $\Ext_R^{\gg0}(N,M)=0$.
\end{cor}

\begin{proof}
By symmetry, we may assume that $\qpd_R(M)<\infty$.
Since $R$ is Gorenstein, $\Gdim_R(M)<\infty$, and $R$ itself is a dualizing complex.
Apply \Cref{symm0} with $D_R=R$.
\end{proof}

\begin{ac}
The author is grateful to Olgur Celikbas for helpful discussions.
In particular, the idea for the proof of \Cref{qcx} arose from these discussions.
The author thanks Paulo Martins for suggesting that the result now stated as \Cref{symm} could be generalized to the form established in \Cref{symm0}.
The author also thanks Akiyoshi Sannai for pointing out an error in the stated hypotheses of Corollary~4.3 in an earlier version of this paper.
\end{ac}


\begin{thebibliography}{99}

\bibitem{AMMRSL}
{\sc T. Ajani, P. Martins, V. D. Mendoza Rubio, and A. J. Soto Levins},
Generalized symmetry in the vanishing of Ext,
preprint (2026), \href{https://arxiv.org/abs/2607.12957}{arXiv:2607.12957v1}.

\bibitem{AB}
{\sc M. Auslander and M. Bridger},
Stable module theory,
{\em Mem. Amer. Math. Soc.} \textbf{94},
Amer. Math. Soc., Providence, R.I., 1969.

\bibitem{IFR}
{\sc L. L. Avramov},
Infinite free resolutions,
in {\em Six lectures on commutative algebra (Bellaterra, 1996)},
Progr. Math. 166,
Birkh\"auser, Basel, 1998, 1--118.

\bibitem{AGP}
{\sc L. L. Avramov, V. N. Gasharov and I. V. Peeva},
Complete intersection dimension,
{\em Inst. Hautes \`Etudes Sci. Publ. Math.} \textbf{86} (1997), 67--114.

\bibitem{AIL}
{\sc L. L. Avramov, S. B. Iyengar and J. Lipman},
Reflexivity and rigidity for complexes, I: Commutative rings,
{\em Algebra Number Theory} \textbf{4} (2010), no.~1, 47--86.


\bibitem{Ber1}
{\sc P. A. Bergh},
Modules with reducible complexity,
{\em J. Algebra} \textbf{310} (2007), 132--147.

\bibitem{Ber2}
{\sc P. A. Bergh},
Modules with reducible complexity, II,
{\em Comm. Algebra} \textbf{37} (2009), 1908--1920.

\bibitem{CP}
{\sc O. Celikbas and G. Piepmeyer},
Syzygies and tensor product of modules,
{\em Math. Z.} \textbf{276} (2014), no.~1--2, 457--468.

\bibitem{CHY}
{\sc H. Chen, J. Hu and X. Yang},
Quasi-projective dimensions of complexes over rings,
preprint (2026), \href{https://arxiv.org/abs/2604.09279}{arXiv:2604.09279v1}.

\bibitem{Gbook}
{\sc L. W. Christensen},
{\em Gorenstein dimensions},
Lecture Notes in Mathematics 1747,
Springer-Verlag, Berlin, 2000.

\bibitem{CFH}
{\sc L. W. Christensen, H. B. Foxby and H. Holm},
{\em Derived category methods in commutative algebra},
Springer Monographs in Mathematics, Springer, Cham, 2025.

\bibitem{CJ}
{\sc L. W. Christensen and D. A. Jorgensen},
Tate (co)homology via pinched complexes,
{\em Trans. Amer. Math. Soc.} \textbf{366} (2014), no.~2, 667--689.

\bibitem{CJ2}
{\sc L. W. Christensen and D. A. Jorgensen},
Vanishing of Tate homology and depth formulas over local rings,
{\em J. Pure Appl. Algebra} \textbf{219} (2015), 464--481.

\bibitem{DFG}
{\sc S. Dey, L. Ferraro, and M. Gheibi},
Complete intersection and quasi-homological dimensions
preprint (2026), \href{https://arxiv.org/abs/2609.26686}{arXiv:2609.26686v1}.

\bibitem{DGI}
{\sc W. G. Dwyer, J. P. C. Greenlees and S. B. Iyengar},
Finiteness in derived categories of local rings,
{\em Comment. Math. Helv.} \textbf{81} (2006), no.~2, 383--432.

\bibitem{FI}
{\sc H.-B. Foxby and S. Iyengar},
Depth and amplitude for unbounded complexes,
in {\em Commutative algebra (Grenoble/Lyon, 2001)},
Contemp. Math. 331, Amer. Math. Soc., Providence, RI, 2003, 119--137.

\bibitem{FL}
{\sc L. Ferraro and J. Lyle},
The derived depth formula for modules of finite quasi-projective dimension,
{\em Bull. Lond. Math. Soc.} \textbf{58} (2026), no.~9, e70491.

\bibitem{GJT}
{\sc M. Gheibi, D. A. Jorgensen and R. Takahashi},
Quasi-projective dimension,
{\em Pacific J. Math.} \textbf{312} (2021), no.~1, 113--147.

\bibitem{Iye}
{\sc S. B. Iyengar},
Depth for complexes, and intersection theorems,
{\em Math. Z.} \textbf{230} (1999), 545--567.

\bibitem{Jor}
{\sc P. J{\o}rgensen},
Symmetry theorems for Ext vanishing,
{\em J. Algebra} \textbf{301} (2006), no.~1, 224--239.

\bibitem{JPMMR}
{\sc V. H. Jorge-P\'erez, P. Martins and V. D. Mendoza-Rubio},
Remarks on Auslander's depth formula for quasi-projective dimension,
{\em Rev. Real Acad. Cienc. Exactas Fis. Nat. Ser. A-Mat.} 
\textbf{120} (2026), 86.

\bibitem{Ker}
{\sc B. Keller},
Deriving DG categories,
{\em Ann. Sci. Ecole Norm. Sup. (4)} \textbf{27} (1994), no.~1, 63--102.

\bibitem{Pol}
{\sc J. Pollitz},
The derived category of a locally complete intersection ring,
{\em Adv. Math.} \textbf{354} (2019), 106752.

\bibitem{Sad}
{\sc A. Sadeghi},
Symmetry in vanishing of Tate cohomology over Gorenstein rings, 
preprint (2017), \href{https://arxiv.org/abs/1608.04588}{arXiv:1608.04588v3}. 

\bibitem{SW}
{\sc S. Sather-Wagstaff},
Complete intersection dimensions for complexes,
{\em J. Pure Appl. Algebra} \textbf{190} (2004), no.~1--3, 267--290.


\bibitem{SY}
{\sc T. Sharif and S. Yassemi},
Special homological dimensions and intersection theorem,
preprint (2004), \href{https://arxiv.org/abs/math/0404182}{arXiv:math/0404182v1}.

\bibitem{Levins}
{\sc A. J. Soto Levins},
The Auslander bound for complexes,
{\em Arch. Math. (Basel)} \textbf{126} (2026), no.~1, 29--39.


\bibitem{Vel}
{\sc O. Veliche},
Gorenstein projective dimension for complexes,
{\em Trans. Amer. Math. Soc.} \textbf{358} (2006), no.~3, 1257--1283.

\bibitem{Wei}
{\sc C. A. Weibel},
{\em An introduction to homological algebra},
Cambridge Studies in Advanced Mathematics 38, Cambridge University Press, Cambridge, 1994.

\end{thebibliography}
\end{document}